\title{A Random Walk on the Category of Finite Abelian $p$-groups}

\documentclass{article}

\usepackage{ifxetex,ifluatex}
\if\ifxetex T\else\ifluatex T\else F\fi\fi T%
  \usepackage{fontspec}
\else
  \usepackage[T1]{fontenc}

  \usepackage{blkarray, bigstrut}
  \usepackage[utf8]{inputenc}
  
  \usepackage{amsthm}
  \usepackage{thmtools}
  \usepackage{mathtools}

  \usepackage{lmodern}
  \usepackage{amssymb}
  \usepackage{amsfonts}
  \usepackage{tikz-cd}
  \usepackage{amsthm}
  \usepackage{xfrac}
  \usepackage{mathtools}
  \usepackage{multirow}
  \usepackage{mathrsfs}
  \usepackage{comment}
  \newtheorem{theorem}{Theorem}[section]
  \newtheorem{lemma}[theorem]{Lemma}
  \newtheorem*{claim}{Claim}

  \newtheorem{sublemma}{Lemma}[lemma]
  \newtheorem*{corollary}{Corollary}

  \newcommand{\probP}{\text{I\kern-0.15em P}}
  \newcommand{\probE}{\text{I\kern-0.15em E}}
  \newcommand{\coker}{\text{coker}}
  \newcommand{\defeq}{\overset{\mathrm{def}}{=\joinrel=}}

  \numberwithin{equation}{section}

  \theoremstyle{remark}
  \newtheorem*{remark}{Remark}
  \theoremstyle{remark}
  \newtheorem*{example}{Example}

  \newtheorem*{definition}{Definition}

  \newcommand{\Q}{\mathbb{Q}}
  \newcommand{\Z}{\mathbb{Z}}
  \newcommand{\fr}{\frac}

  \newcommand{\Hom}{\text{Hom}}
  \excludecomment{mysection}
  \excludecomment{mymysection}
  \excludecomment{maybeinclude}

  \newcommand*{\tempstar}{\multicolumn{1}{|c}{*}}
  
  \newcommand*{\tempvdots}{\multicolumn{1}{|c}{\vdots}}
  \newcommand*{\tempzero}{\multicolumn{1}{|c}{0}}

  \newcommand{\dk}{d^k}
  \newcommand{\GB}{G/\alpha(B)}
  \newcommand{\dkplusone}{d^{k+1}}
  \newcommand{\mo}{\text{\textbf{Moment}}}
  \newcommand{\del}{\delta}
  \newcommand{\Xzero}{X_0}
  \newcommand{\Elambda}{E_{\lambda}}
  \newcommand{\hmap}{h}
  \newcommand{\phimap}{\phi}

  \newenvironment{diagram}
{ \[
\begin{tikzcd}}
{ \end{tikzcd}
\] }

    \newenvironment{mat}
    {
    \left[
\begin{array}}
{ \end{array}
\right]
}

\newcommand{\probM}{\mathcal{M}}

\newcommand{\R}{\mathbb{R}}

\usepackage{mdframed}

\newtheorem*{theorem*}{Theorem}
\newenvironment{ftheo*}
  {\begin{mdframed}\begin{theorem*}}
  {\end{theorem*}\end{mdframed}}

\newcommand{\ssfrac}[2]{#1 \Big/ #2}

\newtheorem{Corollary}{Corollary}[theorem]

\newtheorem*{goal}{Goal}

  \usepackage{enumitem}

   \newcommand{\doper}{d}

   \newcommand{\theoremintro}{theorem}
   
   \DeclareSymbolFont{bbold}{U}{bbold}{m}{n}
   \DeclareSymbolFontAlphabet{\mathbbold}{bbold}
   \newcommand{\one}{\mathbbold{1}}
   
   \theoremstyle{remark}

\fi

 \theoremstyle{definition}

 \newenvironment{eq}
  {\begin{equation}}
  {\end{equation}}

 \newcommand{\emptydot}{\, \cdot \,}

 \newcommand{\questioneq}{\overset{\mathrm{?}}{=\joinrel=}}

 \newcommand{\comm}[1]{} 
 \newcommand{\mmarginpar}{\comm} 
 \newcommand{\F}{\mathbb{F}}
 \newcommand{\sur}{twoheadrightarrow}
 \newcommand{\lone}{L^1(X_0)}
 \newcommand{\ltwoCL}{L^2(X_0,\mu_0)}
 \newcommand{\hil}{\mathcal{H}}

 \newcommand{\llambda}{T}
 \newcommand{\ltwostar}{L^2(X_0,\mu_0)^{*}}
 
 \newcommand{\et}{\emptydot}
 
 \newcommand{\dstar}{d^{*}}
 \newcommand{\dkstar}{d^{*}{}^k}

 \newcommand{\lb}{\Big<}
 \newcommand{\rb}{\Big>}
 \newcommand{\prob}{\probP}
 \newcommand{\lltwoCL}{\ltwostar}
 \newcommand{\Ds}{\Delta_0}
 \newcommand{\Hil}{\mathcal{H}}
 \usepackage{cite}
 \newcommand{\commm}[1]{}
 \newcommand{\marginparr}{\commm}
 \newcommand{\chapter}{section}
 \newcommand{\fouriermap}{\mathcal{U}}

  \usepackage{
  hyperref,
  cleveref,
  }

\author{Nikita Lvov\footnote{nikita.lvov@mail.mcgill.ca}}

\begin{document} 

\maketitle
\begin{abstract}
We study an irreducible Markov chain on the category of finite abelian $p$-groups, whose stationary measure is the Cohen-Lenstra distribution. This Markov chain arises when one studies the cokernel of a random matrix $M$, after conditioning on a submatrix of $M$. We show two surprising facts about this Markov chain. Firstly, it is reversible. Hence, one may regard it is a random walk on finite abelian $p$-groups. The proof of reversibility also explains the appearance of the Cohen-Lenstra distribution in the context of random matrices. Secondly, we can explicitly determine the eigenvalues and eigenfunctions of the infinite transition matrix associated to this Markov chain.
\end{abstract}

\begin{mysection}
\section{Introduction}

\subsubsection{Short Intro on Cohen-Lenstra}


\subsubsection{$\Delta_0$}


\subsubsection{A unitary transform}


\subsubsection{Main Theorem}


\section{Outline}

\subsubsection{Notation}

\subsubsection{Equivalence of definitions of $\Delta_0$}

\subsubsection{Relation to matrices and composition property}

\subsubsection{Spectrum}


\subsubsection{Moment Basis}

\subsubsection{Action of $\Delta_0$ on moments}

\subsubsection{Derivation of Curious Formula}

\subsubsection{Set of Eigenfunctions}

\subsubsection{Orthogonal Complement of Image of Fourier is Kernel of Delta}

\subsubsection{Specializing to $\F_p$. (Optional)}

\section{The operator $\Delta_0$}

\section{Relation to Random Matrices}

\section{Spectrum}

\subsection{The Spectral Theorem for $\Delta_0$}

\subsubsection{Boundedness of Operator}

\subsection{A Basis of Moments}

\subsection{Derivation of Main Formulae}

\subsection{Image of $\fouriermap$}

\end{mysection}

\section*{Introduction}
\begin{mysection}
The cokernel of a large $p$-adic random matrix $M$ is a random abelian $p$-group. Friedman and Washington showed that its distribution asymptotically tends to the well-known Cohen-Lenstra distribution. In this paper, we study an irreducible Markov chain on the category of finite abelian $p$-groups, whose stationary measure is the Cohen-Lenstra distribution. This Markov chain arises when one studies the cokernel of $M$, after conditioning on a submatrix of $M$.  We show two surprising facts about this Markov chain. Firstly, it is reversible. Hence, one may regard it is a random walk on finite abelian $p$-groups.  Secondly, we can explicitly determine the spectrum of the infinite transition matrix associated to this Markov chain. In doing so, we show an interesting identity that generalizes the well-known fact that all the moments of the Cohen-Lenstra distribution are $1$.
\end{mysection}

The Cohen-Lenstra distribution is a probability distribution 
on finite abelian $p$-groups. It is the distribution that 
assigns to each finite abelian $p$-group $G$ a probability 
inversely proportional to the number of automorphisms of that 
group:
\begin{equation}
\label{eqn:cl}
\probP(G) \propto \frac{1}{\#Aut(G)}
\end{equation}

It originated in number theory, in the study of the 
statistical behaviour of class groups or quadratic fields. In \cite{CohenLenstra} 
Cohen and Lenstra made the conjecture that, for odd $p$, the $p$-part of the 
class group of a quadratic imaginary fields is distributed 
exactly in such a way, as the discriminant ranges over all 
possible negative values.

The prediction of Cohen and Lenstra agrees with numerical evidence and a small number of cases have been proven;\commm{the first instance was the result, proven by Davenport and Heilbronn before the conjecture was formulated, that the average number of $3$-torsion elements is in agreement with the value expected from the conjectures.  For other results,} see, for example, the survey \cite{WoodICM}. \commm{starting with the work of Davenport and Heilbronn, (preceding the conjecture)
who were able to compute the average number of 3-torsion elements
in class groups of quadratic fields matches the prediction of
the conjecture; see, for example, the survey \cite{WoodICM}.} A natural question to ask is why this particular distribution (\ref{eqn:cl}) is to be expected. Is this a special property of class groups of quadratic fields? Or rather is it a property of finite abelian $p$-groups, that when such groups appear at random, without any additional structure, they tend to be distributed according to the Cohen-Lenstra measure? In 1989, Friedman and Washington gave credence to the latter point of view, by finding the occurrence of this measure in another context: cokernels of random matrices.

  \begin{\theoremintro} \cite{FriedmanWashington}
  \label{thm:matricesandcl}
  If $\probM_{n,n}$ is a random matrix whose entries are 
independent uniformly distributed elements of $\Z_p$,
  \[
  \lim_{n \rightarrow \infty}\probP\Big( coker(\probM_{n,n}) 
\cong G \Big) \propto \frac{1}{|Aut(G)|}
  \]
  That is, the distribution of the cokernel of 
  $\probM_{n,n}$ is the Cohen-Lenstra distribution in the 
limit $n \rightarrow \infty$.
  \end{\theoremintro}

Moreover, as found in the more recent work of Maples, Wood, 
Nguyen and others, the conclusion of \autoref
{thm:matricesandcl} continues to hold when we replace 
"uniformly distributed" by "identically distributed", 
indicating that the distribution (\ref{eqn:cl}) is indeed 
somehow universal. These results are reminiscent of the central limit theorem.
In recent years, there has been a flurry of work devoted to 
proving more precise and more general statements in this 
direction, for example \cite{Maples1}, \cite{WoodMoments}, \cite{WoodIntegral}, \cite{WoodNguyen}, \cite{WoodNguyen2}, \cite{CheongYu}, \cite{LeeHermitian}, \cite{Yan}. Again, we refer the reader to the survey \cite{WoodICM}.

\begin{mysection}
 as well as... 
\end{mysection}

   In this paper, we study the Cohen-Lenstra distribution from 
a somewhat different point of view. Taking the random matrix 
model of \autoref{thm:matricesandcl} as our starting point, we 
study the distribution of the cokernel of a random 
matrix, after conditioning on one of its submatrices. 

   The first interesting statement we obtain on this line of 
inquiry is that this gives rise to an irreducible Markov chain 
on the category of finite abelian p-groups, whose stationary 
measure is the Cohen-Lenstra distribution. The description of 
this Markov chain is the content of the first section. The generator of this 
Markov chain is what we call the \textit{Cohen-Lenstra operator}, and denote as $\Delta_0$. 

This Markov chain has surprising properties. Firstly, it is 
reversible. Hence, it can be described \commm{realized?} as a random walk on a 
weighted graph whose vertices are finite abelian $p$-groups. 
In the second section, we prove this statement and give a 
simple explicit description of the graph, its edges and edge weights.

   Finally, the property of a Markov chain being reversible is 
equivalent to its generator being a self-adjoint operator, 
with respect to a natural Hilbert space structure. Therefore, the spectral theorem \cite[IV.194]{BourbakiTS}, applies to $\Delta_0$. In the final chapter, we determine the spectrum of $\Delta_0$ explicitly. Our study of the spectrum culminates in the following result:

\begin{theorem*}[Main Theorem]
There is an \textit{explicit} unitary operator $\fouriermap$, such that:
\begin{itemize}
\item The image of $\fouriermap$ contains the image of $\Delta_0$.
\item $
\fouriermap^{-1} \Delta_0 \, \fouriermap
$ is the diagonal operator $|\#G|^{-1}$.
\end{itemize}
\end{theorem*}
\begin{mysection} The spectrum of $\Delta_0$, acting on measures is the spectrum of ... Moreover, the two operators are related by an explicit unitary transformation. 
\end{mysection}

In the remaining part of the introduction, we go through the 
contents of each section in more detail, and give exact statements of our results.

\subsection{Survey of Section 1: Some random operators on $p$-groups}

\subsubsection{The Cohen-Lenstra operator, $\Delta_0$, and two related operators}
   First result: We consider two sets, $X_0$ and $X_1$.
\begin{enumerate}
\item $X_0$ is the set of finite $p$-groups; i.e. $\Z_p$-modules with $\Q_p$ rank 0.
\item $X_1$ is the set of $\Z_p$-modules with $\Q_p$ rank 1.
\end{enumerate}
   We write $G$ to denote a typical element of $X_0$, and $H$, to denote a typical element of $X_1$.
      \newline
   \newline

   \begin{definition}
   There is a random operator, $d$, from $X_1$ to $X_0$ defined as follows: 
  \[ d(H) \] is the quotient of $H$ by a randomly chosen element of $H$.
   \end{definition}

   \begin{definition}
   There is a random operator, $d^{*}$, from $X_0$ to $X_1$ defined as follows: \[ d^{*}(G)\] is the extension of $G$, corresponding to a randomly chosen element of $Ext(G,\Z_p)$.
   \end{definition}
   
   \begin{definition}
   There is a random operator, $\Delta_0$, from $X_0$ to $X_0$ defined as follows: 
   \[
   \Delta_0 \defeq d d^{*}  
   \]
   \end{definition}
   
   $\Delta_0$ generates a Markov chain on finite abelian $p$-groups; we call $\Delta_0$ the \textit{Cohen-Lenstra operator}.

   \subsubsection{Relation with random matrices} 
   Below, we give the main results of section $1$. In the statements below, we will write $*$ to denote independent Haar-random variables valued in $\Z_p$. 
   \begin{itemize}
   \item Let $M_{n,n}$ be an $n \times n$ matrix with cokernel $G$. $d^{*}(G)$ is the cokernel of the random matrix
   \begin{equation}
   \label{eqn:intro:matrixone}
   \begin{mat}{ccc}
   & & \\
   & M_{n,n}& \\
   & & \\ \hline
   * & \hdots & *
   \end{mat}
   \end{equation}

   \item Let $M_{n+1,n}$ be an $n + 1 \times n$ matrix with cokernel $H$. $d(H)$ is the cokernel of the random matrix
   \begin{equation}
   \label{eqn:intro:matrixtwo}
   \begin{mat}{ccc|c}
   &  & & * \\
   & M_{n+1,n} & &\vdots \\ 
   &  & & *
   \end{mat}
   \end{equation}

   \item Let $M_{n,n}$ be an $n \times n$ matrix with cokernel $G$. $\Delta_0(G)$ is the cokernel of the random matrix
   \begin{equation}
   \label{eqn:intro:matrixthree}
   \begin{mat}{ccc|c}
   & & & *\\
   & M_{n,n}& & \vdots \\
   & & & *\\ \hline
   * & \hdots & * & *
   \end{mat}
   \end{equation}
   \end{itemize}

\subsubsection{Relation with previous work}

The idea of associating to a random matrix the random process consisting of the cokernels of its submatrices appears in \cite{Maples1} and \cite{Maples2}. This was seen as an approach to universality. Indeed, universality was also the initial motivation in this thesis for studying this process.

\indent The relation of $d$ to random matrices is an obvious fact, an immediate consequence of the definitions. The relation of $d^{*}$ to random matrices is perhaps not as evident. However, the author would not venture to claim that it is original.

\indent We remark that there is another Markov chain which is often mentioned in the context of the Cohen-Lenstra measure, that is introduced in \cite{Evans}. Given a random $p$-group $\mathcal{G}$ distributed according to the Cohen-Lenstra measure, Evans studied the random process
\[
\mathcal{X}_k \defeq \text{dim}_{\mathbb{F}_p}\Big( \mathbb{F}_p \otimes p^k\mathcal{G} \Big)
\]
and showed that it is a Markov chain. 

\indent Also, in recent work, studying $p$-adic random matrices from a somewhat different perspective, Assiotis introduced a Markov chain in \cite[12]{Theo}. The relation of this latter Markov chain to ours is unclear.

   \subsection{Survey of Section 2: Properties of $\dstar$ and $\doper$ and the reversibility of the Markov chain induced by $\Delta_0$}
   The main result of this chapter is that $\Delta_0$ gives rise to a reversible Markov chain, with stationary measure $\mu_0$.
   \[
   \mu_0(G_1)\probP(G_1 \xrightarrow{\Delta_0} G_2) = 
   \mu_0(G_2)\probP(G_2 \xrightarrow{\Delta_0} G_1) \hspace{0.2in} \forall G_1, G_2 \in X_0
   \]  
   This statement is equivalent to $\Delta_0$ being self-adjoint with respect to a certain natural inner product. We prove this by showing that $\Delta_0$ is the composition of $d$ and its adjoint:
   
   \subsubsection{$d$ and $d^{*}$ are adjoint operators}

   In a natural sense, $d^{*}$ is the formal adjoint of $d$. To define what this means, we need to specify measures on the discrete sets $X_0$ and $X_1$.
   
   \begin{enumerate}
   \item The measure on $X_0$ is $\mu_0$: 
   \begin{equation}
   \label{eqn:cnaught}
   \mu_{0}(G)=\frac{c_0}{|Aut(G)|}
   \end{equation}
   where $c_0$ is normalized so that $\mu_0$ is a probability measure.
   \item For $H \in X_1$, let $H_{tors}$ denote the torsion part of $H$. The measure on $X_1$ is $\mu_1$:
   $$\mu_{1}(H)=\frac{c_1}{|H_{tors}||Aut(H_{tors})|}$$
   where $c_1$ is normalized so that $\mu_1$ is a probability measure.
   \end{enumerate}
   
   These measures induce inner products:
   
   \begin{enumerate}
   \item For any two measures $\nu_1$ and $\nu_2$ on $X_0$,
   \begin{equation}
   \label{eqn:innerproductone}
    \lb \nu_1, \nu_2 \rb_{X_0} = \sum_{G \in X_0} \frac{\nu_1(G)\nu_2(G)}{\mu_0(G)}
   \end{equation}
   \item For any two measures $\nu_1$ and $\nu_2$ on $X_1$, 
   \begin{equation}
   \label{eqn:innerproducttwo}
   \lb \nu_1, \nu_2 \rb_{X_1} = \sum_{H \in X_1} \frac{\nu_1(H)\nu_2(H)}{\mu_1(H)} 
   \end{equation}
   \end{enumerate}

   With these preliminaries we have the following theorem:
   \begin{theorem}
   With respect to the inner products (\ref{eqn:innerproductone}) and (\ref{eqn:innerproducttwo}) $d$ is the adjoint of $d^{*}$:
   \[
   \lb \nu_1, \doper \nu_2 \rb_{X_1}=\lb \dstar \nu_1, \nu_2 \rb_{X_0}
   \]
    Equivalently, for any $G \in X_0$ and $H \in X_1$
   \begin{equation}
   \label{eqn:intro:reversible}
   \mu_0(G) \probP (G \xrightarrow{d^{*}} H) = \mu_1(G) \probP (H \xrightarrow{d} G)
   \end{equation}
   \end{theorem}
    
   \begin{corollary}
   This yields two self-adjoint operators:
   \begin{itemize}
   \item The operator $d d^{*}$ is self-adjoint with respect to $\mu_0$.
   \item The operator $d^{*} d$ is self-adjoint with respect to $\mu_1$.
   \end{itemize}
   \end{corollary}
   The first of these operators is the operator $\Delta_0$, whose study is the main subject of this paper.

\subsubsection{Alternate Perspective} 

We give here an alternative perspective on the equality (\ref{eqn:intro:reversible}):
\[
\mu_0(G) \probP(G \xrightarrow{\dstar} H) = \mu_1(H) \probP(H \xrightarrow{\doper} G)
\]
The pair of operators 
\begin{eq}
\label{eqn:markov}
(\doper, \dstar)
\end{eq}
together form a single Markov chain, whose state space is the union of state spaces, $X_0 \cup X_1$. The equality (\ref{eqn:intro:reversible}) is then equivalent to the statement that this Markov chain is reversible with stationary measure 
\[
(\mu_0,\mu_1)
\] This will be a consequence of the following lemma:

\begin{lemma}
\label{cor:weightedgraph:intro}
We can represent the Markov chain $(\doper,\dstar)$ on $X_0 \cup X_1$, as a random walk on a weighted graph whose edges are exact sequences:
\begin{equation}
\label{eqn:exactsequence:one}
0 \rightarrow \Z_p \rightarrow H \rightarrow G \rightarrow 0
\end{equation}
with weight:
\[
\frac{1}{|Aut(\Z_p \rightarrow H)||G|}.
\]
\end{lemma}
$Aut(\Z_p \rightarrow H)$ denotes the group of automorphisms of $H$ that preserve the homomorphism $\Z_p \rightarrow H$.

\begin{mysection}   
   \begin{theorem*}(\ref{cohenlenstraconvergencetheorem})
   Given any probability measure $\nu$ on $ \in X_0$, we have the pointwise limit:
   \begin{equation}
   \label{eqn:intro:limit}
   \lim_{n \to \infty} \Delta_0^k \nu = \mu_0
   \end{equation}
   \end{theorem*}

The motivation behind this chapter was to show the limit (\ref{eqn:intro:limit}). This can be translated into a statement about random matrices which is doubtless not original. However, the other statements in this chapter have apparently not been previously considered.
\end{mysection}

\subsubsection{Relation with previous work}

   The original motivation behind this chapter was to explain why $\mu_0$ is a stationary measure for $\Delta_0$. This fact can be translated into a statement about random matrices which is not original. However, the other statements in this chapter have apparently not been previously considered.
   
\indent It is worth noting that Markov chains that are self-adjoint with respect to their stationary measure, such as $\Delta_0$, are usually called \textit{reversible} Markov chains.
\begin{mymysection}
Usually called or called?
\end{mymysection}

   \subsection{Survey of Section 3: Eigenfunctions of $\Delta_0$}
   \label{eigenfunctionsection}
   Before stating our theorems about $\Delta_0$, we first introduce a useful collection of measures on $X_0$.

   \begin{definition}
   For a finite abelian group $F$, we let $\text{\textbf{Moment}}[F]$ denote the measure on $X_0$, whose value on $G$ is
   \[
   \# Sur(G,F) \mu_0(G).
   \]
   \end{definition}

   Interestingly, the action of $\Delta_0$ on $\text{\textbf{Moment}}[F]$ is particularly easy to describe.

   \begin{theorem}
   \label{basis}
   \[
   \Delta_0 \left( \text{\textbf{\emph{Moment}}}[F] \right)=\frac{1}{|F|}\sum_{Hom(\Z_p,F)} \text{\textbf{\emph{Moment}}}[\coker(\Z_p \rightarrow F)]
   \]
    \end{theorem}
    
    \begin{definition}
    We define a partial ordering on finite abelian groups, as follows: $F' \leq F$ if and only if $F$ admits a surjection to $F'$.
    \end{definition}
   
   \begin{corollary} (of \autoref{basis})
   The action of $\Delta_0$ on $\text{\textbf{\emph{Moment}}}[\et]$ is upper-diagonal with respect to the above partial ordering.
   \end{corollary}
   
   In other words, for every $F$, we have
     \[
     \Delta_0 \left( \text{\textbf{\emph{Moment}}}[F] \right)=\sum_{F' \leq F}     b_{F'} \text{\textbf{\emph{Moment}}}[F']
     \]
     for some coefficients $b_{F'}$.

   \begin{definition}
   We will say that a finite linear combination of $\text{\textbf{Moment}}[\et]$:
   \[
   \sum_{i} a_{i} \text{\textbf{Moment}}[F_i]
   \]
   has leading term 
   \[
   a_{k} \text{\textbf{Moment}}[F_k]
   \]
   if $F_i \leq F_k$ for all $i$.
   \end{definition}

   \begin{corollary} (of \autoref{basis})
   For every finite abelian group $F$, there exists an eigenfunction $E_{F}$ of $\Delta_0$ that is a finite linear combination of terms of the form $\text{\textbf{\emph{Moment}}}[\emptydot]$ with leading term $\text{\textbf{\emph{Moment}}}[F]$. The eigenvalue associated to $E_F$ is $\frac{1}{|F|}$. The $E_F$ are linearly independent.
   \end{corollary}

   We recall that there is an inner product on $X_0$, associated to $\mu_0$, which we denoted as \[
   \lb \emptydot,\emptydot \rb_{X_0}.
   \]

   \begin{definition}
   We call a measure $\nu$ on $X_0$ \textbf{square-summable} iff its norm under this inner product is finite. The Hilbert space of all square-summable functions will be denoted as $\ltwostar$.
   \end{definition}

   It is simple to verify that the $E_{F}$ are square-summable. The first main theorem of this chapter is the following.
   
   \begin{theorem}
   \label{thm:eigenfunctions:intro}
   Any square-summable eigenfunction of $\Delta_0$ either lies in $ker(\Delta_0)$, or is a finite linear combination of the eigenfunctions $E_{F}$. 
   \end{theorem}

   \subsubsection{Alternate Perspective}

There is an interesting perspective on the results of this chapter. The theorems above can be expressed as a relation between the operator $\Delta_0$ and an \textit{explicit} unitary operator. This is reminiscent of the relation between the one-dimensional Laplacian and the Fourier transform $\mathcal{F}$ in analysis:
\[
\frac{d^2}{dt^2} \mathcal{F}\Big[ f(x) \Big]= \mathcal{F}\Big[ (-x^2) f(x) \Big]
\]
First, we define a unitary operator on square-summable measures.

\begin{lemma}
There is a unique linear continuous operator $\ltwostar \rightarrow \ltwostar$ that takes the measure
\[
\frac{\# Sur ( F , \emptydot)}{\#Aut(\emptydot)}
\]
to the measure
\[
\sqrt{c_0} \frac{\# Sur ( \emptydot, F)}{\#Aut(\emptydot)}
\]
where $c_0$ is the normalization constant (\ref{eqn:cnaught}), above. This operator is \textbf{unitary}.
\end{lemma}

\begin{remark}
The existence of this operator and its unitarity both follow from the curious formula below:
\[
c_0 \sum_G \frac{\# Sur ( G, F_1)\# Sur ( G, F_2)}{\# Aut(G)}= 
\]
\[
=\sum_G \frac{\# Sur ( F_2 , G)\# Sur (F_1, G)}{\# Aut(G)} \hspace{0.5in}
\forall \, F_1, F_2
\]
For the proof of this formula, see \autoref{thm:curioussum:mainbody}.
\end{remark}

Let us denote this linear operator as $\fouriermap$. We have

\begin{theorem}[First Main Theorem]
\label{thm:eqnfundamental:intro}
 $\fouriermap$ satisfies the relation:
\begin{equation}
\label{eqn:fundamental:intro}
\Delta_0 \, \fouriermap \, ( \, f  \, ) = \fouriermap \Big( |\#G|^{-1} f  \Big)
\end{equation}
\end{theorem}

The second main theorem is:

\begin{theorem}[Second Main Theorem]
\label{lem:surjectivitydelta:intro}
 \begin{equation}
 \label{eqn:surjectivitydelta:intro}
  im(\fouriermap) = ker(\Delta_0)^{\perp}.
 \end{equation}
\end{theorem}

\begin{corollary}
\[ im(\Delta_0) \in im(\fouriermap)\]
\end{corollary}

Together, these two theorems subsume \autoref{thm:eigenfunctions:intro} above.

   \subsubsection{Relation with previous work}
   
   To the author's knowledge, the questions and answers in this chapter are new. However, these results are also related to the problem of determining a probability measure from its moments, that has been the subject of recent work of Wood, Sawin, and others (see \cite{WoodICM} and \cite{SawinWood}).
   \begin{mymysection}
"is discussed in"? The stuff concerning measures
\end{mymysection}

\subsubsection{Notation}

Here, we collect some notation used throughout this text.

\begin{itemize}
\item $X_0$ denotes the set of finite abelian $p$-groups, i.e. finitely generated $\Z_p$-modules of $\Q_p$-rank $0$.
\item $X_1$ denotes the set of finitely generated $\Z_p$-modules of $\Q_p$-rank $1$.
\item $G$ denotes a typical element of $X_0$.
\item $H$ denotes a typical element of $X_1$.
\item $\mu_0$ denotes the Cohen-Lenstra measure on $X_0$:
\[
\mu_0(G) \defeq \frac{c_0}{|Aut(G)|}
\]
where $c_0$ is a normalization constant explicitly given by
\[
\prod_{i=1}^{\infty}(1-\frac{1}{p^i})
\]
\item $\mu_1$ denotes the Cohen-Lenstra measure on $X_1$:
\[
\mu_1(H) \defeq \frac{c_1}{|Aut(H_{tors})|
|H_{tors}|}
\]
where $c_1$ is a normalization constant explicitly given by
\[
\prod_{i=2}^{\infty}(1-\frac{1}{p^i})=\frac{p}{p-1}c_0
\]
\item $\doper$ denotes a linear map from probability measures on $X_1$ to probability measures on $X_0$ defined as follows:
\[
\doper(H)
\]
is the quotient of $H$ by a random element chosen from the Haar measure on $H$.
\item $\dstar$ denotes a linear map from probability measures on $X_0$ to probability measures on $X_1$ defined as follows:
\[
\dstar(G)
\]
is a uniformly random $\Z_p$ extension of $G$, i.e. the extension corresponding to a uniformly random element of the finite group $Ext(G,\Z_p)$.
\item $\Delta_0 \defeq \doper \dstar$
\begin{mysection}
\item $X_k$ denotes the set of finitely generated $\Z_p$-modules of $\Q_p$-rank $k$.
\item $\mu_k$ denotes the Cohen-Lenstra measure on $X_k$:
\[
\mu_k(H) \defeq \frac{c_k}{|Aut(H_{tors})|
|H_{tors}|^k}
\]
\marginparr{Check This.}
where $c_k$ is a normalization constant explicitly given by
\[
c_k = \prod_{i=k+1}^{\infty}(1-\frac{1}{p^i})
\]
\item $\dk$ denotes a linear map from probability measures on $X_k$ to probability measures on $X_0$ defined as follows:
\[
\dk(H)
\]
is the quotient of $H$ by a random element of $Hom(\Z_p^k,H)$ chosen from the Haar measure on $Hom(\Z_p^k,H)$.
\item $\dkstar$ denotes a linear map from probability measures on $X_0$ to probability measures on $X_k$ defined as follows:
\[
\dkstar(G)
\]
is a uniformly random $\Z_p^k$ extension of $G$, i.e. the extension corresponding to a uniformly random element of the finite group $Ext(G,\Z_p^k)$.
\end{mysection}
\end{itemize}

\subsubsection{Acknowledgments}

The author would like to thank his advisor, Manjul Bhargava. The author is still fascinated by how much beautiful mathematics originated from two questions that he asked at the beginning of this project. The author would also like to thank Matthew Hernandez, Adam Marcus, Ramon Van Handel, Roger Van Peski, Melanie Matchett Wood and Alexander Yu for useful discussions. The author would also like to particularly thank Alexander Yu for reading earlier versions of this manuscript and making many useful comments. Finally, the author would like to thank Yakov Sinai for his encouragement.

\newpage

\section{Some operators arising in the study of random matrices}
\subsection{The operators $\doper$, $\dstar$ and $\Delta_0$}
\label{doperdstar}

We consider two sets, $X_0$ and $X_1$.
\begin{enumerate}
\item $X_0$ is the set of finite $p$-groups; i.e. $\Z_p$-modules with $\Q_p$-rank 0.
\item $X_1$ is the set of $\Z_p$-modules with $\Q_p$-rank 1.
\end{enumerate}
   To denote an element of $X_0$, we will use the letter $G$, and to denote an element of $X_1$, we will use the letter $H$.

   \begin{definition}
   There is a random operator, $d$, from $X_1$ to $X_0$ defined as follows: 
   \[ d(H)
   \] is the quotient of $H$ by a uniformly random element.
   \end{definition}

   \begin{definition}
   There is a random operator, $d^{*}$, from $X_0$ to $X_1$ defined as follows: \[ d^{*}(G)\] is the extension
   of $G$, corresponding to a uniformly random element of $Ext(G,\Z_p)$.
   \end{definition}

   \begin{definition}
   We also define:
   \[
   \Delta_0 \defeq \doper \dstar
   \]
   $\Delta_0$ is a random operator from $X_0$ to $X_0$
   \end{definition}

   \subsection{Connection of operators $\doper$ and $\dstar$ with random matrices}
   \label{sec:randomoperatorsandrandommatrices}
   Here, we will show how these operators are related to random matrices.
   \begin{remark}
   In the statements below, we will write $*$ to denote independent Haar-random variables valued in $\Z_p$. 
   \end{remark}   
   
   \begin{theorem}
   \label{thm:operatorsandrandommatrices}
   \begin{enumerate}[label=(\alph*)]
   \item Let $M_{n,n}$ be an $n \times n$ matrix with cokernel $G$. $d^{*}(G)$ is distributed as the cokernel of the random matrix:
   \[
   \begin{mat}{ccc}
   & & \\
   & M_{n,n}& \\
   & & \\ \hline
   * & \hdots & *
   \end{mat}
   \]
   \item  Let $M_{n+1,n}$ be an $(n + 1) \times n$ matrix with cokernel $H$. $d(H)$ is distributed as the cokernel of the random matrix:
   \[
   \begin{mat}{ccc|c}
   &  & & * \\
   & M_{n+1,n} & &\vdots \\ 
   &  & & *
   \end{mat}
   \]
   \item Let $M_{n,n}$ be an $n \times n$ matrix with cokernel $G$. $\Delta_0(G)$ is distributed as the cokernel of the random matrix:
   
   \[
   \begin{mat}{ccc|c}
   & & & *\\
   & M_{n,n}& &\vdots\\
   & & \\ \hline
   * & \hdots & * & *
   \end{mat}
   \]
   \end{enumerate}
   \end{theorem}

   The relation for $\Delta_0$ follows immediately from the previous two relations. In the appendix to this chapter, we will prove the relations for $\doper$ and $\dstar$.

   \subsection{$\dk$ and $\dkstar$}

   Below, we will slightly generalize the situation considered in \S \ref{doperdstar}. We now consider the sets, $X_0$ and $X_k$.
\begin{enumerate}
\item As before, $X_0$ is the set of finite $p$-groups; i.e. $\Z_p$-modules with $\Q_p$-rank $0$.
\item $X_k$ is the set of $\Z_p$-modules with $\Q_p$-rank $k$.
\end{enumerate}
   To denote an element of $X_0$, we will use the letter $G$, and to denote an element of $X_k$, we will use the letter $H$.
   \newline
   \newline

   \begin{definition}
   There is a random operator, $\dk$, from $X_1$ to $X_0$ defined as follows: $$\dk(H)$$ is the quotient of $H$ by a randomly chosen element of $Hom(\Z_p^k,H)$
   \end{definition}

   \begin{definition}
   There is a random operator, $\dkstar$, from $X_0$ to $X_k$ defined as follows: $$\dkstar(G)$$ is the extension of $G$, corresponding to a randomly chosen element of $Ext(G,\Z_p^k)$.
   \end{definition}

   \begin{definition}
   We also define
   \[
   \Delta_0^{\Z_p^k}(G) \defeq \dkstar \dk
   \]
   This is a random operator from $X_0$ to $X_0$.
   \end{definition}

    \subsubsection{Connection of $\dk$ and $\dkstar$ with random matrices}

   Here, again, we will show how the operators introduced above are related to random matrices.
   \begin{lemma}
   \begin{enumerate}[label=(\alph*)]
   \item Let $M_{n,n}$ be an $n \times n$ matrix with cokernel $G$. $\dkstar(G)$ is the cokernel of the random $(n + k) \times n$ matrix 
   \[
   \begin{mat}{ccc}
   & & \\
   & M_{n,n}& \\
   & & \\ \hline
   * & \hdots & * \\
   \vdots & \ddots & \vdots \\
   * & \hdots & *
   \end{mat}
   \]

   \item  Let $M_{n+k,n}$ be an $(n + k) \times n$ matrix with cokernel $H$. $\dk(H)$ is the cokernel of the $(n + k) \times (n +k) $ random matrix
   \[
   \begin{mat}{ccc|ccc}
   &  & & * & \hdots & * \\
   & M_{n+k,n} & & \vdots & \ddots & \vdots \\ 
   &  & & * & \hdots & *
   \end{mat}
   \]
   \item Let $M_{n,n}$ be an $n \times n$ matrix with cokernel $G$. $\Delta_0^{\Z_p^k}(G)$ is the cokernel of the $(n + k) \times (n +k) $ random matrix
   
   \[
   \begin{mat}{ccc|ccc}
   &  & & * & \hdots & * \\
   & M_{n,n} & & \vdots & \ddots & \vdots \\ 
   &  & & * & \hdots & * \\ \hline
   * & \hdots & * & * & \hdots & * \\
   \vdots & \ddots & \vdots & \vdots & \ddots & \vdots \\
   * & \hdots & * & * & \hdots & *
   \end{mat}
   \]
   \end{enumerate}
   \end{lemma}

   The relation for $\Delta_0^{\Z_p^k}(G)$ follows immediately from the previous two relations. The proof of the relations for $\dk$ and $\dkstar$ is analogous to the proof of the relations for $\doper$ and $\dstar$ and is therefore omitted.
\marginparr{Proof exists in older version}

\subsection{Appendix}

In the appendix, we give the proof of \autoref{thm:operatorsandrandommatrices}. Firstly, to prove part (b), we note that the cokernel of 
   \[
   \begin{mat}{ccc|c}
   &  & & * \\
   & M_{n+1,n} & &\vdots \\ 
   &  & & *
   \end{mat}
   \]
is the quotient of $coker(M_{n+1,n})$ by a Haar random element of $coker(M_{n+1,n})$.

The remainder of the appendix is devoted to giving some basic facts about $\Z_p$-extensions and using them to prove part (a) of \autoref{thm:operatorsandrandommatrices}.

\subsubsection{Preliminaries on $\Z_p$-extensions}
\label{sec:charactersandextensions}

\begin{lemma}
\label{lem:charactersandextensions}
Given a finite group $G$, there is a bijection between $Hom(G, \Q_p/\Z_p)$ and $ Ext(G,\Z_p)$. Under this bijection, $\phi \in Hom(G, \Q_p/\Z_p)$ corresponds to the extension given by the fiber product:
\begin{diagram}
G \arrow[dr,"\phi"]& \times & \Q_p \arrow[dl] \\ 
&\Q_p / \Z_p& \\
\end{diagram}
Moreover, as a group, the extension corresponding to $\phi$ is isomorphic to 
\[
ker(\phi) \times \Z_p
\]
\end{lemma}

\begin{proof}
Let 
\begin{equation}
\label{eqn:terminalelement}
0 \rightarrow \Z_p \rightarrow \Q_p \rightarrow \Q_p /\Z_p \rightarrow 0
\end{equation}
be the obvious exact sequence.
Let $G$ be a finite group. Given any $\Z_p$-extension of $G$,
\begin{equation}
\label{eqn:exactsequence}
0 \rightarrow \Z_p \rightarrow H \rightarrow G \rightarrow 0,
\end{equation}
there is a unique commutative diagram:
\begin{equation}
\label{eqn:commdiagram}
\begin{tikzcd}
0 \arrow[r] & \Z_p \arrow[r] \arrow[d, equal] & H \arrow[r] \arrow[d, "\psi"] & 
G \arrow[d, "\phi"] \arrow[r] & 0 \\
0 \arrow[r] & \Z_p \arrow[r] & \Q_p \arrow[r] & \Q_p/\Z_p \arrow[r] & 0 \\
\end{tikzcd}
\end{equation}
In other words, (\ref{eqn:terminalelement}) is a terminal element in the category of $\Z_p$-extensions of finite groups. The second square in (\ref{eqn:commdiagram}) is Cartesian. This shows the first part of the lemma.

For the second part, we again refer to the diagram (\ref{eqn:commdiagram}). Note that $im(\psi)$ is isomorphic to $\Z_p$. Furthermore, $ker(\psi)$ isomorphic to the $ker(\phi)$. Since there are no non-trivial extensions of $\Z_p$,
\[
H \cong im(\psi) \times ker(\psi) \cong \Z_p \times ker(\phi)
\]
\end{proof}

\subsubsection{Proof of \autoref{thm:operatorsandrandommatrices}}

Now, we prove part (a) of \autoref{thm:operatorsandrandommatrices}.

\begin{theorem}
\label{thm:randomextclass}
Suppose that $M_{n,n}$ is an invertible matrix. Then, the extension class of
\[
0 \rightarrow \Z_p \rightarrow coker
\left(
\begin{array}
{ccc}
   & & \\
   & M_{n,n}& \\
   & & \\ \hline
   * & \hdots & *
\end{array}
\right) \rightarrow coker(M_{n,n}) \rightarrow 0
\]
is a uniformly random element of $Ext\Big(coker(M_{n,n}),\Z_p \Big)$.
\end{theorem}

To prove this, for any fixed row vector $v$ we will explicitly compute the extension class of
\[
0 \rightarrow \Z_p \rightarrow coker(M') \rightarrow coker(M_{n,n}) \rightarrow 0
\]
\[
M' \defeq \left(
\begin{array}
{ccc}
   & & \\
   & M_{n,n}& \\
   & & \\ \hline
    & v & 
\end{array} \right)
\]

\begin{lemma}
\label{lem:extclass}
The extension class of 
\[
0 \rightarrow \Z_p \rightarrow coker(M') \rightarrow coker(M_{n,n}) \rightarrow 0
\]
 corresponds to the following element of $Hom\Big( coker(M_{n,n}, \Q_p/\Z_p) \Big)$:
\[
g \mapsto - v \, M_{n,n}^{-1} \, g \mod \Z_p 
\]
\end{lemma}

\begin{proof}
For the proof we need to find explicit $\psi$ and $\phi$ to complete the diagram:
\begin{diagram}
0 \arrow[r] & \Z_p \arrow[r] \arrow[d, equal] & coker(M') \arrow[r] \arrow[d, "\psi"] & 
coker(M_{n,n}) \arrow[d, "\phi"] \arrow[r] & 0 \\
0 \arrow[r] & \Z_p \arrow[r] & \Q_p \arrow[r] & \Q_p/\Z_p \arrow[r] & 0 \\
\end{diagram}
This is an exercise in linear algebra.
\end{proof}

\begin{proof} (of \autoref{thm:randomextclass})
\autoref{thm:randomextclass} is a corollary of \autoref{lem:extclass}. Indeed, if $v$ is sampled from a uniformly random distribution on $\Z_p^n$, then 
\[
g \mapsto - v \,  M_{n,n}^{-1} \, g \mod \Z_p 
\]
is a uniformly random element of $Hom\Big( coker(M_{n,n}, \Q_p/\Z_p) \Big)$.
\end{proof}

\newpage

\section{Properties of $\doper$ and $\dstar$ and the reversibility of $\Delta_0$}
\begin{mysection}

\end{mysection}

\begin{mysection}
\marginparr{Maybe talk about why the moments are what they are (in your notes).}
\end{mysection}
\begin{mysection}
In this section, we will investigate properties of the operators we defined, and in particular prove the reversibility of $\Delta_0$. 
\end{mysection}
\subsection{Survey of Section 2}

We study two operators on groups, that were introduced in the previous \chapter, $\doper$ and $\dstar$.

\begin{mysection}
We briefly recall their definitions:
\begin{itemize}
\item Given a finite $p$-group $G$, $\dkstar(G)$ is a uniformly random $\Z_p^k$ extension of $G$.
\item Given a $p$-group $H$ of $\Q_p$-rank $k$, $\dk(H)$ is the quotient by a uniformly random  element $\Hom(\Z_p^k,H)$.
\end{itemize}
\end{mysection}

One of the main results of this section is

\begin{theorem}
$\dstar$ is the adjoint of $\doper$ with respect to the measures $\mu_0$ and $\mu_1$.
\end{theorem}

\begin{corollary}
The operator $\Delta_0 = \doper \dstar$ is self-adjoint with respect to the measure $\mu_0$.
\end{corollary}

\begin{mysection}
 which we recall are defined as:

\begin{itemize}
\item $\mu(G) \propto \frac{1}{|Aut(G)|}$
\item $\mu(G) \propto \frac{1}{|Aut(H_{tors})||H_{tors}|}$
\end{itemize}
\end{mysection}

The exact expression we need to show is:

\begin{equation}
\label{eqn:reversibility}
\mu(G) \probP(G \xrightarrow{\dstar} H) = \mu(H) \probP(H \xrightarrow{\doper} G) 
\end{equation}
for all $G, H$

\begin{remark}
A concrete way to interpret this equality is as follows. We can get a measure on pairs $(G,H) \in X_0 \times X_1$ in two ways:
\begin{enumerate}
\item $\doper(\mu_1)$, i.e. first sampling from $\mu_1$ and then applying the random operator $\doper$.
\item $\dstar(\mu_0)$, i.e. first sampling from $\mu_0$ and then applying the random operator $\dstar$.
\end{enumerate}
The equality (\ref{eqn:reversibility}) implies that the two models give the same distribution on $X_0 \times X_1$
\end{remark}

We will in fact prove a slightly stronger statement:
\begin{theorem}
\label{thm:exactseq}
$\doper(\mu_1)$ and $\dstar (\mu_0)$ induce the same measure on exact sequences:
\[
0 \rightarrow \Z_p \rightarrow H \rightarrow G \rightarrow 0
\]

Furthemore, this measure is very explicit: the measure of the exact sequence
\[
0 \rightarrow \Z_p \rightarrow H \rightarrow G \rightarrow 0
\]
is proportional to 
\[
\frac{1}{|G||Aut(\Z_p \rightarrow H)|}
\]
\end{theorem}

The main tool in the proof of \autoref{thm:exactseq} is the orbit-stabilizer theorem.

\begin{mysection}
\begin{corollary}
OPTIONAL: $\Delta_0$ can be modeled as a random walk on the graph: 
\[
...
\]
Maybe put this at the beginning??
\end{corollary}
\end{mysection}

\subsubsection{Alternate Perspective} 

We give here an alternative perspective on the equality (\ref{eqn:reversibility}):
\[
\mu(G) \probP(G \xrightarrow{\dstar} H) = \mu(H) \probP(H \xrightarrow{\doper} G)
\]

\begin{remark}
The pair of operators 
\begin{eq}
\label{eqn:markov:intro}
(\doper, \dstar)
\end{eq} together form a single Markov chain, whose state space is the union of state spaces, $X_0 \cup X_1$. (\ref{eqn:reversibility}) is then equivalent to the statement that this Markov chain is \textit{reversible}.
\end{remark}

A prototypical example of a \textit{reversible} Markov chain is a random walk on a graph. 

\begin{remark}
A reversible Markov chain with a countable state space is \textit{reversible} if and only if it can be represented as a random walk on a \textit{graph with weighted edges}.
\end{remark}

\begin{lemma}
\label{cor:weightedgraph}
We can represent the Markov chain $(\doper,\dstar)$ on $X_0 \cup X_1$, as a random walk on a weighted graph whose edges are exact sequences:
\[
0 \rightarrow \Z_p \rightarrow H \rightarrow G \rightarrow 0
\] 
with weight:
\[
\frac{1}{|Aut(\Z_p \rightarrow H)||G|}.
\]
\end{lemma}

\autoref{cor:weightedgraph} is in fact a consequence of \autoref{thm:exactseq}.

\subsubsection{Proof that explicit graph gives the Markov Chain $(\dstar,\doper)$}

In this survey, we describe a direct proof of \autoref{cor:weightedgraph}, bypassing \autoref{thm:exactseq}. This will be a little more stream-lined than the proof of the latter, while containing all the main ingredients.

In order to show that the Markov chain can represented as a random walk on the graph in \autoref{cor:weightedgraph}, we need to show that the transition probabilities coincide.

To show that the transition probabilities for $\dstar$ coincide with the transition probabilities of our graph, we have to show the following theorem:

\begin{theorem}
\label{thm:gmeasure}
Suppose we are given $G$ and we generate a random exact sequence by picking an element of $Ext(G,\Z_p)$ \textit{uniformly} at random. 
Then the probability of an \textit{isomorphism class} of exact sequences
\[
0 \rightarrow \Z_p \rightarrow H \rightarrow G \rightarrow 0
\]
 is inversely proportional to $|Aut(\Z_p \rightarrow H)|$:
\[
\probP\Big( 0 \rightarrow \Z_p \rightarrow H \rightarrow G \rightarrow 0 \Big) \propto \frac{1}{|Aut(\Z_p \rightarrow H)|}
\]
\end{theorem}

\begin{proof}
The probability is proportional to size of the orbit under the action of $Aut(G)$. The theorem then follows the orbit-stabilizer theorem, which tells us that the size of each orbit is:
\[
\frac{|Aut(G)|}{|Aut(\Z_p \rightarrow H)|}
\]
\end{proof}

\begin{theorem}
\label{thm:hmeasure}
Suppose we are given $H$ and we generate a random exact sequence
\[
0 \rightarrow \Z_p \xrightarrow{h} H \rightarrow H/h \rightarrow 0
\]
by picking $h \in H$ from the Haar measure on $H$.
Then the probability of each isomorphism class of exact sequences is inversely proportional to:
\[
|Aut(\Z_p \xrightarrow{h} H)||H/h|
\]
\end{theorem}
In order to show the theorem, we need to show that the measure of the orbit $Aut(H) \circ h$ in $H$ is proportional to 
\[
\frac{1}{|Aut(\Z_p \xrightarrow{h} H)||H/h|}
\] up to an overall constant that may depend on $H$, but not on $h$.

\begin{example}
One can first of all verify that this holds in the simplest case, for example when $H \cong \Z_p$ and $h$ is any element. Then $Aut(H) \circ h$ is the set of all elements of $\Z_p$ that have the same $p-adic$ norm as $h$. The proportion of these elements in $H$ is \[\frac{p-1}{p}\frac{1}{|H/h|}
\]
\end{example}

\begin{proof}
We outline the general proof. We use an orbit-stabilizer formula for actions of infinite groups. 
\begin{itemize}
\item The map $Aut(H) \circ h \rightarrow H$ is a local isomorphism from $Aut(H)$ to $H$.
\item To get the measure of the orbit, we calculate the change-of-measure factor (the "Jacobian") from $Aut(H)$ to $H$. 
\item Then we integrate over $Aut(H)$ and divide by the size of the stabilizer of $h$.
\end{itemize}
\end{proof}

\autoref{thm:gmeasure} and \autoref{thm:hmeasure} imply that the explicit graph in \autoref{cor:weightedgraph} represents the Markov chain $(\doper,\dstar)$.

\begin{corollary}
The Markov chain $(\doper,\dstar)$ is reversible.
\end{corollary}

\subsubsection{Conclusion}

We have outlined the proofs of the fact that the Markov chain $(\dstar,\doper)$ is reversible with respect to \textit{some measure}. We can conclude that this measure is $(\mu_0,\mu_1)$ by keeping track of the proportionality constants in the proofs of \autoref{thm:gmeasure} and \autoref{thm:hmeasure}, outlined above.

The final result is the following
\begin{theorem}
The Markov chain is reversible with respect to the measure $(\mu_0,\mu_1)$, or equivalently,
\[
\mu_0(G) \probP( G \xrightarrow{\dstar} H ) = \mu_1(H) \probP(H \xrightarrow{\doper} G)
\]
\end{theorem}

\subsubsection{The operators $\dkstar$ and $\dk$}

In the previous section, we defined $\dkstar$ and $\dk$, as follows:

\begin{definition}
Given a finite $p$-group $G$, $\dkstar(G)$ is a uniformly random $\Z_p^k$ extension of $G$.
\end{definition}

\begin{definition}
Given a $p$-group $H$ of $\Q_p$-rank $k$, $\dk(H)$ is the quotient by a uniformly random  element $\Hom(\Z_p^k,H)$.
\end{definition}

\begin{claim}
The same arguments as we used to prove $\dstar=\doper^T$ can be used to show that $\dkstar=\dk{}^T$
\end{claim}

\subsubsection{Composability}

Suppose that, as usual $*$ represents independent uniformly distributed random variables.
In the previous section we have shown that
\[
\doper \dstar \big( \coker(M) \big) = \coker
\begin{mat}{ccc|c}
 &  & & * \\
 & M & & \vdots \\
 &  & & * \\ \hline
* & \hdots &* & * \\
\end{mat}
\]
and that 
   \[
   \dk \dkstar \big(\coker(M) \big) = \coker
   \begin{mat}{ccc|ccc}
   &  & & * & \hdots & * \\
   & \probM_{n,n} & & \vdots & \ddots & \vdots \\ 
   &  & & * & \hdots & * \\ \hline
   * & \hdots & * & * & \hdots & * \\
   \vdots & \ddots & \vdots & \vdots & \ddots & \vdots \\
   * & \hdots & * & * & \hdots & *
   \end{mat}
   \]
Reiterating the first equality $k$ times, gives
\[
(\doper \dstar)^k \big( \coker(M) \big) = \dk \dkstar \big(\coker(M) \big)
\]
for any square matrix $M$.

Every finite abelian $p$-groups can be represented as $coker(M)$, for some square matrix $M$. Hence, we can conclude:
\begin{equation}
\label{eqn:composability}
(\doper \dstar)^k   = \dk \dkstar 
\end{equation}

\begin{mysection}
\subsubsection{Convergence to the Cohen-Lenstra measure}
As an application, we describe one of many possible proofs of the following statement

\begin{claim}
$\Delta_0^N(\nu)$ converges to the Cohen-Lenstra measure as $N \rightarrow \infty$ for any initial probability measure $\nu$.
\end{claim}

There are several ways to show this. In particular, it can be deduced from two results obtained in this section:
\[
\dk{}^T = \dkstar 
\]
and 
\[
(\doper \dstar)^k   = \dk \dkstar 
\]
\begin{proof}
The deduction is given in the last section of this chapter. 
\end{proof}
\end{mysection}

\subsection{Organization}

In the ensuing sections, we give detailed proofs of the statements outlined above.
\commm{... ... ... ...}

\subsection{Measures and inner products on $X_0$, $X_1$ and $X_k$}
\label{sec:measures}
In this section, we define measures on the spaces $X_0$, $X_1$ and $X_k$:
 \begin{enumerate}
   \item The measure on $X_0$ is $\mu_0$: 
   $$\mu_{0}(G)=\frac{c_0}{|Aut(G)|}$$
   where $c_0$ is normalized so that $\mu_0$ is a probability measure. Explicitly:
   \[
   c_0 = \prod_{i=1}^{\infty}(1-\frac{1}{p^{i}})
   \]
   \item For $H \in X_1$, let $H_{tors}$ denote the torsion part of $H$. The measure on $X_1$ is $\mu_1$:
   $$\mu_{1}(H)=\frac{c_1}{|H_{tors}||Aut(H_{tors})|}$$
   where $c_1$ is normalized so that $\mu_1$ is a probability measure. Explicitly:
   \[
   c_1 = \prod_{i=2}^{\infty}(1-\frac{1}{p^{i}})
   \]
   \item For $H \in X_k$, let $H_{tors}$ denote the torsion part of $H$. The measure on $X_1$ is $\mu_1$:
   $$\mu_{k}(H)=\frac{c_k}{|H_{tors}|^k|Aut(H_{tors})|}$$
   where $c_k$ is normalized so that $\mu_k$ is a probability measure. Explicitly:
   \[
   c_k = \prod_{i=k+1}^{\infty}(1-\frac{1}{p^{i}})
   \]
   \end{enumerate}
   \marginparr{Give a justification for the explicit form.}
   
   These measures induce inner products:
   
   \begin{enumerate}
   \item For any two measures $\nu_1$ and $\nu_2$ on $X_0$, $$ \lb \nu_1, \nu_2 \rb_{X_0} = \sum_{G \in X_0} \frac{\nu_1(G)\nu_2(G)}{\mu_0(G)} $$
   \item For any two measures $\nu_1$ and $\nu_2$ on $X_1$, $$ \lb \nu_1, \nu_2 \rb_{X_1} = \sum_{H \in X_1} \frac{\nu_1(H)\nu_2(H)}{\mu_1(H)} $$
   \item For any two measures $\nu_1$ and $\nu_2$ on $X_k$, $$ \lb \nu_1, \nu_2 \rb_{X_k} = \sum_{H \in X_1} \frac{\nu_1(H)\nu_2(H)}{\mu_k(H)} $$
   \end{enumerate}

\subsection{$\doper$ is the adjoint of $\dstar$ with respect to the Cohen-Lenstra measure}
\label{sec2}
\subsubsection{Statement of the Problem}
We aim to prove that $\doper$ is the adjoint of $\dstar$ with respect to the Cohen-Lenstra measures on $X_0$ and $X_1$, which we will here denote as $\mu_0$ and $\mu_1$.

This is the statement that, for any measures $\nu_0$ on $X_0$ and any measure $\nu_1$ on $X_1$,
\[
\lb \nu_1, \dstar \, \nu_0 \rb_{X_1} = \lb \doper \, \nu_1,  \nu_0 \rb_{X_0}
\]

By letting $\nu_0$ and $\nu_1$ be measures concentrated on the groups $G$ and $H$ respectively, this expression becomes
\begin{equation}
\label{selfadj}
\mu_0(G)\probP(G \xrightarrow{\dstar} H ) = \mu_1(H)\probP(H \xrightarrow{\doper} G )
\end{equation}

where $\probP(G \xrightarrow{\dstar} H )$ is the probability of the random process $\dstar$ taking the group $G$ to the group $H$.
\newline

Conversely, a proof of relation \ref{selfadj} also implies self-adjointness by linearity. Thus, in the rest of this section, we will be proving relation \ref{selfadj}.

\subsubsection{Probability measures on exact sequence}

\paragraph{The LHS of (\ref{selfadj})} 

First of all, we make the observation that
\[
\mu_1(H)\probP(H \xrightarrow{\doper} G )
\]
defines a probability measure on $X_0 \times X_1$. This measure is defined as follows:
\begin{itemize}
\item[(A)] Sample $H \in X_1$ from the probability measure $\mu_1$.
\item[(B)] Mod out by a uniformly random element of $H$ to get $G \in X_0$. 
\end{itemize}

\begin{remark}
Steps $(A)$ and $(B)$ define a probability measure on isomorphism classes of exact sequences:
\[
0 \rightarrow \Z_p \rightarrow H \rightarrow G \rightarrow 0
\]
\end{remark}

\begin{lemma}
Under this probability measure,
\[
\probP\Big( 0 \rightarrow \Z_p \xrightarrow{h} H \rightarrow G \rightarrow 0 \Big) 
= \mu_{1}(H) \probP(Aut(H) \circ h)
\]
\end{lemma}

\paragraph{The RHS of \autoref{selfadj}}

Now, we make the observation that
\[
\mu_0(H)\probP(G \xrightarrow{\dstar} H )
\]
also defines a probability measure on $X_0 \times X_1$. This measure is defined as follows:
\begin{itemize}
\item[(A]$'$) Sample $G \in X_0$ from the probability measure $\mu_0$.
\item[(B]$'$) Pick a uniformly random $\Z_p$-extension of $G$ to get $H \in X_1$.
\end{itemize}

\begin{remark}
Steps $A'$ and $B'$ define a probability measure on isomorphism classes of exact sequences:
\[
0 \rightarrow \Z_p \rightarrow H \rightarrow G \rightarrow 0
\]
as follows:

\begin{equation}
\label{eqn:exsequence}
\probP\Big( 0 \rightarrow \Z_p \rightarrow H \rightarrow G \rightarrow 0 \Big) =
\end{equation}
\[
\mu_0(G) \frac{|Aut(G) \circ \phi|}{|Ext(G,\Z_p)|}
\]
where $\phi$ is the element of $Ext(G,\Z_p)$ associated to the exact sequence \ref{eqn:exsequence}.

\end{remark}

\subsubsection{Reformulation}

The relation (\ref{selfadj}) resembles the condition for a Markov chain to be reversible, and indeed it can be interpreted as such.

\begin{definition}
Define $(\doper, \dstar)$ to be the Markov chain on the state space $X_0 \cup X_1$ whose transition probabilities are defined as
\[
\prob\Big(G\xrightarrow{(\doper, \dstar)} H \Big) \defeq \prob\Big(G\xrightarrow{ \dstar } H \Big) 
\text{ for }G \in X_0, \, H \in X_1
\]
\[
\prob\Big(H\xrightarrow{(\doper, \dstar)} G \Big) \defeq \prob\Big(H\xrightarrow{\doper} G \Big)
\text{ for }G \in X_0, \, H \in X_1
\]
\[
0 \text{ otherwise }
\]
\end{definition}

The relation (\ref{selfadj}) is then the statement that $(\doper, \dstar)$ is a reversible Markov chain, with stationary measure $(\mu_0,\mu_1)$.

\subsubsection{Proof that $(\doper)^T=\dstar$}

The main goal of this section is to prove the following theorem:
\begin{theorem}
\label{generalselfadjtheorem}
\begin{equation}
\label{eqn:weightsone}
\mu_{0}(G) \probP (G \xrightarrow{\dstar} H) = \mu_1(H) \probP (H \xrightarrow{\doper} G)
\, \forall \, G,H
\end{equation}
\end{theorem}

\begin{corollary}
We can represent the Markov chain $(\doper,\dstar)$ as a random walk whose edges are pairs $(G,H) \in X_0 \times X_1$, and where the edge $(G,H)$ has weight (\ref{eqn:weightsone}).
\end{corollary}

To every $Aut(H)$ orbit of elements
$$
\hmap \in Hom(\Z_p,H) 
$$
such that
\begin{equation}
\label{eqn:extclass}
0 \rightarrow \Z_p \xrightarrow{h} H \rightarrow G \rightarrow 0
\end{equation}
is exact, we can associate an $Aut(G)$ orbit of elements

\[
\phimap \in Ext(G,\Z_p)
\]
corresponding to the extension class (\ref{eqn:extclass}).

Therefore, to prove \autoref{generalselfadjtheorem}, it is sufficient to prove: 
\begin{theorem}
\label{auxiliarygeneralselfadjtheorem}
\begin{equation}
\label{eqn:weightstwo}
\mu_0(G) \probP \left( Aut(G) \circ \phimap \right) = \mu_1(H) \probP \left( Aut(H) \circ \hmap \right).
\end{equation}
Furthermore, both sides are equal to 
\[
\frac{c_0}{|H/h||Aut(H,h)|}
\]
\end{theorem}

\begin{corollary}
\autoref{generalselfadjtheorem} follows from (\ref{eqn:weightstwo}) by summing over automorphism orbits. 
\end{corollary}

\begin{corollary}
We can represent $(\doper,\dstar)$ as a random walk on the weighted bipartite graph on the vertex set $X_0 \cup X_1$, whose edges are exact sequences
\[
0 \rightarrow \Z_p \xrightarrow{h} H \rightarrow G \rightarrow 0
\]
with weight 
\[
\frac{c_0}{|H/h||Aut(H,h)|}
\]
\end{corollary}

We will prove \autoref{auxiliarygeneralselfadjtheorem} by combining 
\autoref{philemma} and \autoref{hlemma}, below.

\begin{lemma}
\label{philemma}
$$
\probP \left( Aut(G) \circ \phimap \right) =\frac{|Aut(G)|}{|Aut(G,\phimap)||G|}  
$$
\end{lemma}

\begin{proof}

Let $\phimap$ be an element of the $Aut(G)$ orbit of $Ext(G,\Z_p)$

$$
\probP \left( Aut(G) \circ \phimap \right)=\frac{|Aut(G) \circ \phimap|}{|Ext(G,\Z_p)|}=
$$

$$
=\frac{|Aut(G)|}{|Aut(G,\phimap)|} \frac{1}{|Ext(G,\Z_p)|}
$$
by the orbit-stabilizer formula. To conclude the result, we note that 
\[
|Ext(G,\Z_p)|=|Hom(G,\Q_p/\Z_p)|=|G|
\]
\end{proof}

\begin{corollary} (of \autoref{philemma})
\begin{equation}
\label{eqn:philemma}
\mu_0(G) \probP \left( Aut(G) \circ \phimap \right) = \frac{c_0}{|Aut(G,\phimap)||G|}=
\frac{c_0}{|Aut(H,\hmap)||H/h|}
\end{equation}
\end{corollary}

\begin{proof}(of Corollary)
We need to show the second equality in the corollary, i.e.
\[
|Aut(G,\phimap)||G| = |Aut(H,\hmap)||H/h|
\]
\begin{itemize}
\item By definition, $H/h \cong G$. Hence $|G| =|H/h|$
\item Both $Aut(H,\hmap)$ and $Aut(G,\phimap)$ is in bijection with the group of automorphisms of
\begin{diagram}
0 \arrow[r] &\Z_p  \arrow[r,"h"] & H \arrow[r] & G \arrow[r] & 0  \\
\end{diagram}
\begin{mysection}
i.e with diagrams
\begin{diagram}
0 \arrow[r] &\Z_p \arrow[d, equal] \arrow[r,"h"] & H \arrow[d] \arrow[r] & G \arrow[d] \arrow[r] & 0  \\
0 \arrow[r] &\Z_p  \arrow[r,"h"] & H \arrow[r] & G \arrow[r] & 0  \\
\end{diagram}
\end{mysection}

Hence $Aut(H,\hmap) \cong Aut(G,\phimap)$. \marginparr{ADD MORE DETAIL LATER}

\end{itemize}
\end{proof}

\begin{lemma}
\label{hlemma}

\[
\probP \left( Aut(H) \circ \hmap \right) = 
\]

\[
=\frac{|H_{tors}| |Aut(H_{tors})|} {|Aut(H,\hmap)||H/h|} (1-\frac{1}{p})
\]

\end{lemma}

Before giving the proof of \autoref{hlemma}, we give two corollaries:

\begin{corollary}
\begin{equation}
\label{eqn:hlemma}
\mu_{1}(H)\probP \left( Aut(H) \circ \hmap \right) = \frac{c_0}{|Aut(H,\hmap)||H/h|}=
\frac{c_0}{|Aut(G,\phimap)||G|}
\end{equation}
Hence, combining (\ref{eqn:philemma}) and (\ref{eqn:hlemma}), we get \autoref{auxiliarygeneralselfadjtheorem}.
\end{corollary}

The proof of \autoref{hlemma} is a little trickier because $H$ is pro-finite. Let $\nu$ be the Haar probability measure on $Hom(\Z_p,H)$. Tautologically,
\[
\probP \left( Aut(H) \circ \hmap \right) = \nu(Aut(H) \circ \hmap) 
\]

Let $d(Aut(H))$ be the Haar probability measure on $Aut(H)$. $Aut(H,h)$, the stabilizer of $h$, is a finite subgroup of $Aut(H)$, and we have an orbit-stabilizer formula: 
\marginparr{ADD MORE DETAIL}
\[
\nu(Aut(H) \circ \hmap) = \frac{1}{|Aut(H,h)|}\int\limits_{Aut(H)} \frac{dH}{d(Aut(H))}  d(Aut(H))
\]

\autoref{hlemma} would follow from the following.
\begin{lemma}
\label{lem:differential}
\[
\frac{dH}{d(Aut(H))}= \frac{p-1}{p} \frac{|Aut(H_{tors})||H_{tors}|}{| H/h |}
\]
\end{lemma}

We choose a sequence of open neighbourhoods of $0$, $U_l$, defined as follows:

\[
U_l \coloneqq \Big\{ u \in Aut(Z_p)=GL_1(\Z_p) \Big| u \equiv id \mod p^l \Big\}
\]

\begin{remark}
If $l$ is large enough relative to $H$, the $U_l$ are normal subgroups.
\end{remark}

\[
\frac{dH}{d(Aut(H))}=\lim_{l\rightarrow \infty} \frac{\nu(U_l \circ h)}{ |Aut(H)/U_l|^{-1}}
\]

\begin{mysection}
We let $U_\mathcal{I}$ be a neighborhood of the identity in $Aut(H)$, defined as the set of all automorphisms that are congruent to the identity $\mod \mathcal{I}$.

\begin{itemize}
\item The index of $U_\mathcal{I}$ in $Aut(H)$ is $Aut(\Z_p/\mathcal{I}\Z_p) |Aut(H_{tors})| |H_{tors}|$.
\item The index of $U_\mathcal{I} \circ h$ in $H$ is $(h/\mathcal{I}h)(H/h)$
\end{itemize}
\end{mysection}

\begin{mysection}
The subgroup $Aut(H, \hmap)$ is a finite subgroup of $Aut(H)$, and hence it is a discrete subgroup. Therefore, we have:

$$
\nu(Aut(H) \circ \hmap) = \nu(Aut(H) \circ U_l\hmap) =  \frac{|Aut(H)/U_l|}{|Aut(H,\hmap)|}\nu(U_l\hmap)
$$
\end{mysection}

\begin{lemma}
\begin{equation}
\label{sublemmaexpression1}
\nu(U_l \circ \hmap) =\frac{1}{|H/{p^{l} \hmap}|} =\frac{1}{|H/h||\Z/p^l \Z|} 
\end{equation}
\end{lemma}
\marginparr{changed $m$ to $l$. Is this correct?}

\begin{proof}

By the fact that $\nu$ is a Haar measure,

$$
\nu \Big( U_l \circ \hmap \Big)=\nu \Big(   (U_l-1) \circ \hmap \Big) = \nu \Big(p^l \hmap \Big)=
|H/p^{l}h|^{-1}=\frac{1}{|H/h||\Z/p^l \Z|}
$$

\end{proof}

\begin{lemma}
\label{lem:sublemmaexpression2}
For large $l$, the index of $U_l$ in $Aut(H)$ is
\begin{equation}
\label{sublemmaexpression2}
|Aut(H)/U_l|=|H_{tors}| |Aut(H_{tors})| |GL_1(\Z/p^l\Z)|
\end{equation}
\end{lemma}

\begin{corollary}
\[
\frac{\nu(U_l \circ h)}{ |Aut(H)/U_l|^{-1}}
=\frac{p-1}{p}\frac{|H_{tors}| |Aut(H_{tors})|}{|H/h|}
\]
\end{corollary}

The corollary implies
\[
\frac{dH}{d(Aut(H))}
=\frac{p-1}{p}\frac{|H_{tors}| |Aut(H_{tors})|}{|H/h|}
\]
proving \autoref{lem:differential}.

\begin{proof}(Of \autoref{lem:sublemmaexpression2})
In fact, more is true.  $U_l$ is a normal subgroup of $Aut(H)$ and, for large $L$, we have an isomorphism of groups:
\[
Aut(H)/U_l \cong H_{tors} \times Aut(H_{tors}) \times GL_1(\Z/p^l\Z)
\]

We verify this below. First of all,
\[
H \cong H_{tors} \times \Z_p
\]
It is a straightforward consequence of this that:
\[
Aut(H) \cong Aut(H_{tors}) \times Hom(\Z_p, H_{tors}) \times GL_1(\Z_p)
\]
Hence, for large $l$,
\[
Aut(H) \otimes \Z / p^l \Z \cong H_{tors} \times Aut(H_{tors}) \times GL_1(\Z/p^l\Z) 
\]
But $Aut(H) \otimes \Z / p^l \Z$ is precisely
\[
Aut(H) / U_l.
\]
\end{proof}

\marginparr{Here, there is a mysection about isomorphisms of automorphisms groups}
\begin{mysection}
\begin{lemma}
$$
Aut(H,\hmap) \cong Aut(G,\phimap)
$$
\end{lemma} 

\begin{proof}
There is a map from $Aut(H,\hmap)$ to $Aut(G,\phimap)$. It is surjective and we need only show that it is injective. But the difference between the identity 

and any element of 
$$ker\Big( Aut(H,\hmap) \rightarrow Aut(G,\phimap) \Big)$$
must factor through 
$$
Hom(G, \Z_p)
$$
which is $0$.
\end{proof}
\end{mysection}

\subsection{The operator $\Delta_0$}
 
   The fact that $\dstar$ is the adjoint of $\doper$ yields two self-adjoint operators:
   \begin{itemize}
   \item The operator $d d^{*}$ is self-adjoint with respect to $\mu_0$.
   \item The operator $d^{*} d$ is self-adjoint with respect to $\mu_1$.
   \end{itemize}
   
   The first of these operators is what we have called $\Delta_0$.
   \marginparr{The first of these operators is what we have called $\Delta_0$}
   \newline
   \newline
   It is a self-adjoint operator; we will later determine its spectrum.

\subsection{Proof that $(\dk)^T=\dkstar$}

\begin{theorem}
\label{kgeneralselfadjtheorem}
$$
\mu_{0}(G) \probP (G \xrightarrow{\dkstar} H) = \mu_{k}(H) \probP (H \xrightarrow{\dk} G)
$$
\end{theorem}

To every $Aut(H)$ orbit of elements
$$
\hmap \in Hom(\Z_p^k,H) 
$$
such that $\coker(\hmap) \cong G$, we can associate an $Aut(G)$ orbit of elements

$$
\phimap \in Ext(G,\Z_p^k)
$$
such that the extension associated to $\phi$ is isomorphic to $(H, \hmap)$.

Therefore, to prove \autoref{kgeneralselfadjtheorem}, it is sufficient to prove: 
\begin{theorem}
\label{kauxiliarygeneralselfadjtheorem}
\begin{equation}
\label{eqn:weightstwok}
\mu_0(G) \probP \left( Aut(G) \circ \phimap \right) = \mu_k(H) \probP \left( Aut(H) \circ \hmap \right).
\end{equation}
Furthermore, both sides are equal to 
\[
\frac{c_0}{|H/im(h)|^k|Aut(H,h)|}
\]
\end{theorem}

\begin{corollary}
\autoref{kgeneralselfadjtheorem} follows from (\ref{eqn:weightstwok}) by summing over automorphism orbits. 
\end{corollary}

\begin{corollary}
We can represent $(\doper,\dstar)$ as a random walk on the weighted bipartite graph on the vertex set $X_0 \cup X_1$, whose edges are exact sequences
\[
0 \rightarrow \Z_p^k \xrightarrow{h} H \rightarrow G \rightarrow 0
\]
with weight 
\[
\frac{c_0}{|H/im(h)|^k|Aut(H,h)|}=\frac{c_0}{|G|^k|Aut(G,\phi)|}
\]
\end{corollary}

Most of the proofs carry over from the next section and will not be repeated. The only proof that requires modification is the computation of $\probP \left( Aut(H) \circ \hmap \right)$, where $H \cong H_{tors} \times \Z_p^k$ and $\hmap \in Hom(\Z_p^k,H)$.

\begin{lemma}
\[
\nu(  Aut(H)  \circ  h  )  =  \frac{c_0}{c_k} \frac{|Aut(H_{tors})||H_{tors}|^k}{|Aut(H,h)|| H/im(h) |^k} 
\]
\end{lemma}

Again, we can express $\nu( Aut(H) \circ h )$ as an integral:
\[
\nu(Aut(H) \circ \hmap) = \frac{1}{|Aut(H,h)|}\int\limits_{Aut(H)} \frac{dH}{d(Aut(H))}  d(Aut(H))
\]

\autoref{hlemma} would follow from the following.
\begin{lemma}
\label{lem:differentialk}
\[
\frac{dH}{d(Aut(H))}= \frac{c_0}{c_k} \frac{|Aut(H_{tors})||H_{tors}|^k}{| H/h |}
\]
\end{lemma}

First, we show that this equality holds in a neighborhood of the identity of $Aut(H)$.
We choose a sequence of open neighbourhoods of $0$, $U_l$, defined as follows:

\[
U_l \coloneqq \{ u \in Aut(Z_p)=GL_1(\Z_p) \Big| u \equiv id \mod p^l \}
\]

\begin{remark}
If $l$ is large enough relative to $H$, the $U_l$ are normal subgroups.
\end{remark}

\[
\frac{dH}{d(Aut(H))}=\lim_{l\rightarrow \infty} \frac{\nu(U_l \circ h)}{ |Aut(H)/U_l|^{-1}}
\]

As before, we can show that for sufficiently large $l$,
\[
|Aut(H)/U_l|=|Aut(H_{tors})||H_{tors}|^k|GL_k(\Z/p^l\Z)|
\]
However the following lemma requires more work
\begin{lemma}
\[
\nu(U_l \circ h)=\frac{1}{|H/im(h)|^k|M_k(\Z/p^l\Z)|}
\]
\end{lemma}
First we show the following:
\begin{claim}
\[
\nu( h \circ U_l ) = \frac{1}{|H/im(h)|^k|M_k(\Z / p^l\Z)|}
\]
\end{claim}

\begin{proof}
Indeed, 
\[
\nu( h \circ U_l )=\nu( h \circ (U_l-1) )
\]
But the expression on the right is the probability that the image of a uniformly random element of $Hom(\Z_p^k,H)$ factors through $p^l h$.
\begin{itemize}
\item The probability that a uniformly random homomorphism from $\Z_p^k$ factors through $h$ is 
\[
\frac{1}{|Hom(\Z_p^k,H/im(h))|}
\]
\item The probability that a uniformly random homomorphism factors through $p^l h$, \textit{given that it factors through $h$} is:
\[
\frac{1}{|Hom(\Z_p^k,(\Z/p^l\Z)^k)|}=\frac{1}{|M_k(\Z/p^l\Z)|}
\]
\end{itemize}
Hence 
\[
\nu( h \circ (U_l) )= \frac{1}{|H/im(h)|^k|M_k(\Z/p^l\Z)|}
\]
\end{proof}

It remains to show that 
\[
\nu(U_l \circ h) = \nu(h \circ U_l)
\]

\begin{claim}
For sufficiently large $l$, 
$$\nu(U_l \circ \hmap)=\nu(\hmap \circ U_l).$$
\end{claim}

\begin{proof}

This is true if $H_{tors}=0$. In this case, this is the probability that a uniformly chosen random matrix lies inside

$$\nu(U_l \circ \hmap).$$

But then, we can get the result by taking the transpose, and using the fact that

$$
|coker(h^T)|=|coker(h)|
$$

To get the conclusion in the general case, we note that 

$$
H \cong H_{tors} \times \Z_p^k
$$

and that, for large $l$,

$$
\hmap \circ (U_l -  1)
$$

is a set of maps from $\Z_p^k$ to the $\Z_p^k$ factor above.

If we compute

$$
\frac{\nu \Big( \hmap \circ (U_l -  1) \Big)}{\nu \Big( Hom(\Z_p^k, \Z_p^k) \Big)} \fr{\nu(Hom(\Z_p^k, \Z_p^k))}{\nu(Hom(\Z_p^k, H))}
$$

By the above, the left term is 

$$
|\Z_p^k/im(p^{l}\hmap)|^{-k} 
$$

While the right term is

$$
|H/\Z_p^k|^{-k}
$$

Their product is evidently
$$
|H/im(p^{l}\hmap)|^{-k}
$$
\end{proof}

\begin{mysection}
..........................................................
\begin{lemma}
\label{philemma}
$$
\probP \left( Aut(G) \circ \phimap \right) =\frac{|Aut(G)|}{|Aut(G,\phimap)||G|^k}  
$$
\end{lemma}

\begin{proof}

Let $\phimap$ be an element of the $Aut(G)$ orbit of $Ext(G,\Z_p^k)$

$$
\probP \left( Aut(G) \circ \phimap \right)=\frac{|Aut(G) \circ \phimap|}{|Ext(G,\Z_p^k)|}=
$$

$$
=\frac{|Aut(G)|}{|Aut(G,\phimap)|} \frac{1}{|G|^k}
$$
by the orbit-stabilizer formula.
\end{proof}

\begin{lemma}
\label{hlemma}
$$
\probP \left( Aut(H) \circ \hmap \right) = 
$$

$$
=\frac{|Hom(\Z_p^k, H_{tors})| |Aut(H_{tors})|} {|Aut(H,\hmap)||G|^k} \prod_{i=1}^{k}(1-\frac{1}{p^{i}})
$$
\end{lemma}
This is a little trickier as $H$ is profinite. Let $\nu$ be the Haar probability measure on $Hom(\Z_p^k,H)$. Tautologically,
$$
\probP \left( Aut(H) \circ \hmap \right) = \nu(Aut(H) \circ \hmap) 
$$

We choose a sequence of open neighbourhoods of $0$, $U_l$, defined as follows:

$$
U_l \coloneqq \{ u \in Aut(\Z_p^k)=GL_k(\Z_p) \Big| u \equiv id \mod p^l \}
$$

If $l$ is large enough relative to $H$, the $U_l$ are normal subgroups.

The subgroup $Aut(H, \hmap)$ is a finite subgroup of $Aut(H)$, and hence it is a discrete subgroup. Therefore, we have:

$$
\nu(Aut(H) \circ \hmap) = \nu(Aut(H) \circ U_l\hmap) =  \frac{|Aut(H)/U_l|}{|Aut(H,\hmap)|}\nu(U_l\hmap)
$$

\begin{sublemma}
\begin{equation}
\label{sublemmaexpression1}
\nu(U_l\hmap) =\frac{1}{|coker(p^{m} \hmap|^k} = p^{-{lk^2}}\frac{1}{|G|^k} 
\end{equation}
\end{sublemma}

\begin{proof}

First we treat the case when $\hmap$ is on the other side of $U_l$.

By the fact that $\nu$ is a Haar measure,

$$
\nu \Big( \hmap \circ U_l \Big)=\nu \Big( \hmap \circ (U_l-1) \Big)
$$

But this is the probability that the image of a random map $\Z_p^k \rightarrow H$ is contained inside
$$p^l im(h)$$

This probability is equal to

$$
\frac{1}{|coker(p^{l}h)|^{k}}=p^{-lk^2}|coker(h)|^{-k}
$$

It remains to prove the following:
\begin{claim}
For sufficiently large $l$, 
$$\nu(U_l \circ \hmap)=\nu(\hmap \circ U_l).$$
\end{claim}

This is true if $H_{tors}=0$. In this case, this is the probability that a uniformly chosen random matrix lies inside

$$\nu(U_l \circ \hmap).$$

But then, we can get the result by taking the transpose, and using the fact that

$$
|coker(h^T)|=|coker(h)^{\vee}|=|coker(h)|
$$

To get the conclusion in the general case, we note that 

$$
H \cong H_{tors} \times \Z_p^k
$$

and that, for large $l$,

$$
\hmap \circ (U_l -  1)
$$

is a set of maps from $\Z_p^k$ to the $\Z_p^k$ factor above.

If we compute

$$
\frac{\nu \Big( \hmap \circ (U_l -  1) \Big)}{\nu \Big( Hom(\Z_p^k, \Z_p^k) \Big)} \fr{\nu(Hom(\Z_p^k, \Z_p^k))}{\nu(Hom(\Z_p^k, H))}
$$

By the above, the left term is 

$$
|\Z_p^k/im(p^{l}\hmap)|^{-k} 
$$

While the right term is

$$
|H/\Z_p^k|^{-k}
$$

Their product is evidently
$$
|H/im(p^{l}\hmap)|^{-k}
$$
\end{proof}

\begin{sublemma}
For large $l$,
\begin{equation}
\label{sublemmaexpression2}
|Aut(H)/U_l|=|Hom(\Z_p^l, H_{tors})| |Aut(H_{tors})| |GL_k(\Z/p^l\Z)|
\end{equation}
\end{sublemma}

\begin{proof}
The term on the left counts the number of automorphisms of $H / p^l H$ that come from automorphisms of $H$. So first, we need to understand $Aut(H)$:

$$
H \cong H_{tors} \times \Z_p^k
$$
Therefore,
$$
Aut(H) \subset Hom( H_{tors} \times \Z_p^k , H_{tors} \times \Z_p^k) = Hom( H_{tors}, H_{tors} \times \Z_p^k) \times Hom(\Z_p^k , H_{tors} \times \Z_p^k)
$$
$$
=Hom( H_{tors}, H_{tors}) \times Hom(\Z_p^k , H_{tors}) \times Hom(\Z_p^k,\Z_p^k)
$$
In order for a homomorphism in 
$$
Hom( H_{tors}, H_{tors}) \times Hom(\Z_p^k , H_{tors}) \times Hom(\Z_p^k,\Z_p^k)
$$
to lie in $Aut(H)$, it is necessary that it be injective on the left $H_{tors}$ factor
 and surjective on the right $\Z_p^k$ factor.
Therefore,
$$
Aut(H) \subset 
Inj( H_{tors}) \times Hom(\Z_p^k , H_{tors}) \times Sur(\Z_p^k ,\Z_p^k)=Aut( H_{tors}) \times Hom(\Z_p^k , H_{tors}) \times Aut(\Z_p^k , \Z_p^k)
$$
Conversely, any element of the group on the right lies in $Aut(H)$. 
To get the lemma, we observe that for large $l$,
$$
Aut(H) \rightarrow Aut \left( H / p^l H \right)
$$
is a homomorphism with kernel $U_l$. Therefore, \autoref{sublemmaexpression2} follows by reducing modulo $p^l$.
\end{proof}

We multiply \autoref{sublemmaexpression1} and \autoref{sublemmaexpression2} and observe that:
$$
|GL_k(\Z/p^l\Z|p^{-{lk^2}}=\fr{|GL_k(\Z/p^l\Z|}{|Hom\left((\Z/p^l\Z)^k,(\Z/p^l\Z)^k\right)|}
$$
is the probability that a $k \times k$ matrix is invertible. This probability is 
$$
\prod_{i=1}^{k}(1-\frac{1}{p^{i}})
$$

Together, this gives us \autoref{hlemma}.
\newline
\newline
Finally to deduce \autoref{kauxiliarygeneralselfadjtheorem}, we compute the ratio of the quantities in \autoref{philemma} and \autoref{hlemma}:
$$
\frac{|\left( Aut(G) \circ \phimap \right)|} {|\left( Aut(H) \circ \hmap \right)|}
=
$$

$$
=\frac{|Hom(\Z_p^k, H_{tors})| |Aut(H_{tors})|} {|Aut(H,\hmap)||G|^k} \frac{|Aut(G,\phimap)||G|^k}{|Aut(G)|} \prod_{i=1}^{k}(1-\frac{1}{p^{i}})=
$$

$$
= \frac{|Aut(G,\phimap)|}{|Aut(H,\hmap)|} \frac{|H_{tors}|^k |Aut(H_{tors})|}{|Aut(G)|} \fr{c_0}{c_k}
$$
$$
=\frac{|Aut(G,\phimap)|}{|Aut(H,\hmap)|}\fr{\mu_0(G)}{\mu_k(H)}
$$

Thus to conclude the proof of \autoref{kauxiliarygeneralselfadjtheorem}, we need only to show the following isomorphism:
\begin{lemma}
$$
Aut(H,\hmap) \cong Aut(G,\phimap)
$$
\end{lemma} 

\begin{proof}
There is a map from $Aut(H,\hmap)$ to $Aut(G,\phimap)$. It is surjective and we need only show that it is injective. But the difference between the identity
and any element of 
$$ker\left( Aut(H,\hmap) \rightarrow Aut(G,\phimap) \right)$$
must factor through 
$$
Hom(G, \Z_p^k)
$$
which is $0$.
\end{proof}
\end{mysection}

   \subsection{The operator $\Delta_0^{\Z_p^k}$ and  composability}

   The fact that $\dkstar$ is the adjoint of $\dk$ yields two self-adjoint operators:
   \begin{itemize}
   \item The operator $\dk \dkstar$ is self-adjoint with respect to $\mu_0$.
   \item The operator $\dkstar \dk$ is self-adjoint with respect to $\mu_k$.
   \end{itemize}
   We will call the first of these operators $\Delta_0^{\Z_p^k}$.  
   \newline
   \newline
   We will prove in what follows the following composability statement:
   \begin{theorem}
   \label{thm:composability}
   For any $k$ and $l$,
   \[
   \Delta_0^{\Z_p^k} \circ \Delta_0^{\Z_p^l} = \Delta_0^{\Z_p^{k+l}}
   \]
   \end{theorem}
   \begin{Corollary}
   \label{cor:composability}
   \[
   \Delta_0^{\Z_p^k}=\Big( \Delta_0^{\Z_p} \Big)^{k} = \Big( \Delta_0 \Big)^{k}
   \]
   \end{Corollary}

\begin{mysection}
   \begin{remark}
   It is tantalizing to think that there might be a more general family of operators 
   $\Delta_0^G$, parametrized by abelian $p-groups$, such that
   \[
   \Delta_0^{G_1} \circ \Delta_0^{G_2} = \Delta_0^{G_1 \times G_2}
   \]
   However, this would imply that all such operators would commute. Therefore, all such hypothetical $\Delta_0^{G}$ would have the same stationary measure.
   \end{remark} 
\end{mysection}

   \subsubsection{Proof of composability}

   We will first prove the following:

   \begin{lemma}
   \label{firstrepresentabilitylemma}
   For any group $G$, there exists a square matrix $M$ such that $coker(M)=G$.
   \end{lemma}

   \begin{proof}
   There are various ways to see this. For example, this follows from the structure theorem for abelian groups.
   \end{proof}

   \begin{theorem*} (\ref{thm:composability})
   For any $k$ and $l$,
   \[
   \Delta_0^{\Z_p^{k+l}} = \Delta_0^{\Z_p^{l}} \circ \Delta_0^{\Z_p^{k}} 
   \]
   \end{theorem*}

   \begin{proof}
   Consider the nested sequence of random matrices:
   \begin{enumerate}[label=(\alph*)]
   \item The $n \times n$ matrix \[ \probM_{n,n} \]
   \item The $(n+k) \times (n+k)$ matrix   \[
   \begin{mat}{ccc|ccc}
   &  & & * & \hdots & * \\
   & \probM_{n,n} & & \vdots & \ddots & \vdots \\ 
   &  & & * & \hdots & * \\ \hline
   * & \hdots & * & * & \hdots & * \\
   \vdots & \ddots & \vdots & \vdots & \ddots & \vdots \\
   * & \hdots & * & * & \hdots & *
   \end{mat}
   \]
   
   \item The $(n+k+l) \times (n+k+l)$ matrix   \[
   \begin{mat}{ccc|ccc|ccc}
   &  & & * & \hdots & *& * & \hdots & * \\
   & \probM_{n,n} & & \vdots & \ddots & \vdots & \vdots & \ddots & \vdots \\ 
   &  & & * & \hdots & * & * & \hdots & * \\ \hline
   * & \hdots & * & * & \hdots & * & * & \hdots & * \\
   \vdots & \ddots & \vdots & \vdots & \ddots & \vdots & \vdots & \ddots & \vdots  \\
   * & \hdots & * & * & \hdots & * & * & \hdots & * \\ \hline
   * & \hdots & * & * & \hdots & * & * & \hdots & * \\
   \vdots & \ddots & \vdots & \vdots & \ddots & \vdots & \vdots & \ddots & \vdots  \\
   * & \hdots & * & * & \hdots & * & * & \hdots & *
   \end{mat}
   \]
   \end{enumerate}

   \begin{enumerate}[label=(\alph*)]
   \item The cokernel of this random matrix is \[ \coker(\probM_{n,n}) \]
   \item The cokernel of this random matrix is \[ \Delta_0^{\Z_p^k}\coker(\probM_{n,n}) \]
   \item The cokernel of this random matrix is
   \[
   \Delta_0^{\Z_p^{k+l}}\coker(\probM_{n,n}). 
   \] 
   But it is also evidently
   \[ \Delta_0^{\Z_p^{l}} \Delta_0^{\Z_p^{k}}\coker(\probM_{n,n} )\]
   \end{enumerate}
   Hence, we have 
   \begin{equation}
   \label{composabilitytwo}
   \Delta_0^{\Z_p^{l}} \Delta_0^{\Z_p^{k}}(G) = \Delta_0^{\Z_p^{k+l}}(G)
   \end{equation}
   for any $p$-group $G$ that can be represented as the cokernel of a square matrix. 
   By \autoref{firstrepresentabilitylemma}, it follows that \autoref{composabilitytwo} holds for all groups $G$. By linearity, it follows that 
   \[
   \Delta_0^{\Z_p^{l}} \Delta_0^{\Z_p^{k}} = \Delta_0^{\Z_p^{k+l}}
   \]
   holds in general, applied to any probability measure on $X_0$.
   \end{proof}   

\newpage

\section{The Spectrum of $\Delta_0$}
\subsection{Survey}

The main object of study in this section will be $\ltwostar$, together with the action of $\Delta_0$.

\subsubsection{General remarks about $\Ds$ and $\lltwoCL$}

\begin{definition}
We recall the definition of $\lltwoCL$ given in the introduction: $\lltwoCL$ contains all measures $\nu$ on $X_0$, for which
\[
\sum_G \nu(G)^2\#Aut(G) < \infty
\]
$\lltwoCL$ is equipped with the following inner product:
\[
\Big< \nu_1 \, , \nu_2 \Big> \, \defeq \, c_0^{-1} \sum_G \nu_1(G)\nu_2(G)\#Aut(G) =
\]
where $c_0$ is a normalization coefficient, explicitly given by:
\begin{equation}
\label{eqn:cconstant}
c_0 \defeq  \prod_{j=1}^{\infty} (1-\frac{1}{p^j})
\end{equation}
\end{definition}

One of the theorems in the appendix of this section, \autoref{thm:boundedoperator}, states that $\Ds$ is a bounded operator on $\lltwoCL$.

By the results in Section $2$, $\Delta_0$ is reversible. As an immediate consequence,  $\Ds$ is a self-adjoint operator on $\lltwoCL$.

\mmarginpar{This mysection remark might be useful.}
\begin{mysection}
\begin{remark}
It is in fact a compact operator, as will follow from the results of this section.
\end{remark}
\end{mysection}

Self-adjoint operators satisfy a spectral theorem. The main goal of this section is to completely determine the spectrum of $\Ds$.

\subsubsection{Eigenfunctions of $\Ds$: an explicit formula}

Firstly, we will construct a certain set of eigenfunctions of $\Ds$ that lie in $\lltwoCL$. This construction is surprisingly explicit. We will first make a seemingly unrelated aside.

Note that the finitely supported measures:
\[
\frac{\# Sur ( F , \emptydot)}{\#Aut(\et)}
\]
form a basis for all finitely supported measures and hence a basis for $\lltwoCL$.

\begin{lemma}
There is a unique linear continuous operator $\lltwoCL \rightarrow \lltwoCL$ that takes the measure
\[
\frac{\# Sur ( F , \emptydot)}{\#Aut(\et)}
\]
to the measure
\[
\sqrt{c_0} \frac{ \# Sur ( \emptydot, F)}{\#Aut(\et)} 
\]
where $c_0$ is the normalization constant (\ref{eqn:cconstant}), above. This operator is \textbf{unitary}.
\end{lemma}

\begin{remark}
The existence of this operator and its unitarity both follow from the curious formula below:
\[
c_0 \sum_G \frac{\# Sur ( G, F_1)\# Sur ( G, F_2)}{\# Aut(G)}= 
\]
\[
=\sum_G \frac{\# Sur ( F_2 , G)\# Sur (F_1, G)|}{\# Aut(G)} \hspace{0.5in}
\forall \, F_1, F_2
\]
For the proof of this formula, see \autoref{thm:curioussum:mainbody}.
\end{remark}

Let us denote this linear operator as $\fouriermap$. Our first main result is the following:

\begin{theorem}
\label{thm:eqnfundamental}
 $\fouriermap$ satisfies the relation:
\begin{equation}
\label{eqn:fundamental}
\Ds \, \fouriermap \, ( \, f  \, ) = \fouriermap \Big( |\#G|^{-1} f  \Big)
\end{equation}
\end{theorem}

\begin{definition}
We define the measure $1_F$ as the unique measure that is positive, supported on $F$ and has unit norm in $\lltwoCL$. 
\end{definition}
By relation \ref{eqn:fundamental}, $\fouriermap$ takes eigenfunctions of $|\#G|^{-1}$ to eigenfunctions of $\Ds$. The functions $1_F$ form a natural orthonormal basis of eigenfunctions of $|\#G|^{-1}$. The $\fouriermap(1_F)$ are hence orthonormal eigenfunctions of $\Ds$, which span $im(\fouriermap)$.

\subsubsection{The spectrum of $\Ds$ and the image of $\fouriermap$}

Therefore, the set of eigenvalues of $|\#G|^{-1}$ are mapped to a subset of the eigenvalues of $\Ds$. If we knew that $\fouriermap$ was surjective, this would completely determine the spectrum of $\Ds$. We have a partial result in this direction, which is our second main result of this section:

\mmarginpar{Maybe improve the formulation}

\begin{theorem}
\label{lem:surjectivitydelta}
 \[
  im(\fouriermap) = ker(\Ds)^{\perp}.
 \]
\end{theorem}

\begin{corollary}
The operators $\Ds$ and $|\#G|^{-1}$ have the same spectrum.
\end{corollary}

This follows because, by \autoref{lem:surjectivitydelta} and \autoref{thm:eqnfundamental}, the two operators have the same spectrum away from $0$. But $0$ is a limit point of the spectrum of $|\#G|^{-1}$. Since the spectrum is closed, this proves the corollary.

\begin{mysection}
\subsubsection{Question}
\end{mysection}

\begin{remark}
We do not yet know the answer to the following:
\[
ker(\Ds) \cap \lltwoCL \questioneq \emptyset
\]
i.e. we do not yet know whether $\fouriermap$ is surjective. The surjectivity of $\fouriermap$ would imply that the $E_F$ span $\lltwoCL$.

\begin{mysection}
Note that (MAYBE this is unnecessary.)
\[
ker(\Delta_0)= ker(\del_0) 
\]
\end{mysection}

\begin{mysection}
Some other remarks about this question are given at the end of this survey.
\end{mysection}

\end{remark}

\subsubsection{Proof Strategy}

\paragraph{Proof Strategy for \autoref{thm:eqnfundamental}.}
The main ingredient in the proof of \autoref{thm:eqnfundamental} will be \autoref{lem:momentdelta}:

\begin{lemma}
\label{lem:momentdelta}
For every $G$, there exist coefficients $b_F$ such that
\[
\Ds \frac{\# Sur( \emptydot , G)}{\#Aut(\et)} = \frac{1}{|\#G|} \frac{\# Sur( \emptydot , G)}{\#Aut(\et)}+ \sum_{F<G} b_F 
\frac{\# Sur( \emptydot , F)}{\#Aut(\et)}
\]
where $F<G$ means that there exists a surjection from $F$ to $G$.
\end{lemma}

Granting the lemma, the deduction of the first main theorem is short enough to be given here in its entirety. The idea is that we first prove that the action of $\Delta_0$ on $\fouriermap(1_G)$ must be upper-diagonal. Then, by symmetry, we conclude that it is diagonal.

\begin{proof} (of \autoref{thm:eqnfundamental} using \autoref{lem:momentdelta})
\[
\Ds \frac{\# Sur( \emptydot , G)}{\#Aut(\et)} = 
\frac{1}{|\#G|} \frac{\# Sur( \emptydot , G)}{\#Aut(\et)}+ \text{ lower order terms }
\]
Therefore, by the definition of $\fouriermap$,
\[
\Ds \fouriermap \left( \frac{\# Sur( G , \emptydot )}{\#Aut(\et)} \right) = 
\frac{1}{|\#G|} \fouriermap \left( \frac{\# Sur(G , \emptydot )}{\#Aut(\et)}\right) + \text{ lower order terms }
\]
\begin{mysection}
\[
\Ds\Big[ \fouriermap( 1_G) + \text{ lower order terms } \Big]= 
\frac{1}{|\#G|} \fouriermap (1_G) + \text{ lower order terms }
\]
\end{mysection}
Hence,
\[
\Ds \fouriermap( 1_G) = 
\frac{1}{|\#G|} \fouriermap (1_G) + \text{ lower order terms }
\]
Therefore,
\[
\lb \fouriermap (1_F), \Ds \fouriermap (1_G) \rb=
 \lb \Ds \fouriermap (1_F),  \fouriermap (1_G) \rb=
\begin{cases}
0 \text{ if } F<G\\
0 \text{ if } F>G\\
|\#G|^{-1} \text{ if } F=G
\end{cases}
\]

\end{proof}

\paragraph{Proof strategy for \autoref{lem:surjectivitydelta}.}

We will say a few words about the proof strategy for \autoref{lem:surjectivitydelta}: we will repeatedly use the following fact:

\begin{theorem}
\label{thm:spectralinput}
Suppose $T$ is a bounded self-adjoint operator acting on a Hilbert space $\Hil$. If
\[
\lim_{k \rightarrow \infty} T^k ( h ) = 0 \hspace{0.5in} \forall \, h \text{ in a dense subset of $\Hil$}
\]
then the operator norm of $T$ is at most $1$.
\end{theorem}

\autoref{thm:spectralinput} is deduced in the appendix from the Spectral Theorem. We will apply this theorem to the Hilbert space:

\[
\Hil_{\lambda} \defeq span\Big\{ \fouriermap(1_F)   \Big|   |F|<\lambda \Big\}^{\perp}
\]

and the operator $T \defeq \lambda \Ds$. The dense subset of $\Hil$ will be set of finitely supported functions. The hypothesis of \autoref{thm:spectralinput} is verified by \autoref{thm:leadingtermone} below.
\marginparr{A comment about finitely supported functions.}

\begin{mysection}
\paragraph{Proof Strategy: Conclusion.} 
\end{mysection}

Hence, the restriction of $\Ds$ to $\Hil_{\lambda}$ has operator norm at most $ \lambda^{-1}.
$

\begin{definition}
Define 
\[
\Hil_{\infty} \defeq \bigcap_{\lambda} \Hil_{\lambda}
\]
Alternatively, $\Hil_{\infty}$ is the orthogonal complement of the image of $\fouriermap$.
\end{definition}

We can now conclude the following: $\Ds$, acting on $\Hil_{\infty}$ has operator norm at most:  
\[
\inf \lambda^{-1} = 0
\]
 In other words, 
\[
\Hil_{\infty} = ker(\Ds) \cap \lltwoCL,
\]
 This shows \autoref{lem:surjectivitydelta}.

\paragraph{Proof strategy for \autoref{lem:surjectivitydelta}: main technical tool}

\begin{theorem}
\label{thm:leadingtermone}
For any finitely supported measure $\nu$, there exist $\llambda \in \mathbb{R}$ and $a \in \ltwostar$ such that 
\[
\Ds{}^k \nu = \frac{1}{\llambda^k} \fouriermap(a) + o\left(\frac{1}{\llambda^k}\right) 
\]
\end{theorem}

\begin{proof}
The proof exploits the identity for $\Delta_0^k$ given in Chapter 2. See \S \ref{sec:alleigenfunctions}.
\end{proof}

\mmarginpar{Maybe correct references in the preceding "proof" section.}

\begin{mysection}
\subsubsection{Historical Remark}
\begin{remark}
We make an important observation. Chronologically, a variant of \autoref{thm:leadingtermone} was the first theorem to be proven, in a slightly weaker form than given here. This theorem alone was sufficient to determine the spectrum of $\Delta_0$.  
The main observation was that the limit
\[
\lim_{k\rightarrow \infty} \lambda^k \Delta_0^k f,
\]
when it exists, is an eigenfunction of $\Delta_0$ with eigenvalue $\lambda$.
\end{remark}
\end{mysection}

\begin{mysection}
\begin{remark}
To summarize, we can always extract the leading term of $\Delta_0^k f$ when $f$ is finitely supported; the leading term will be an eigenfunction of $\Delta_0$. This fact can be used repeatedly to find eigenfunctions, then pass to the orthogonal complement of the eigenfunctions we found and look at the next leading term...
\end{remark}
\end{mysection}

\subsection{Action of $\Delta_0$ on moments}
\begin{mysection}
In this chapter, we will show that the eigenfunctions of $\Delta_0$ that are square-summable, with respect to the measure $\mu_0$, can all be expressed as finite sums of moment measures.

\begin{remark}
This is reminiscent of how the eigenfunctions of the simple harmonic oscillator can all be expressed as the product of a polynomial and a Gaussian.
\end{remark}
\end{mysection}

Central to the proofs of the first main theorem in this chapter will be a formula for the action of $\Delta_0$ on \textit{moment measures}.

We first recall the definition of moment measures: 
\marginparr{We first recall the definition of moment measure from the introduction.}

\begin{definition}
   For an abelian group $F \in \Xzero$, we let $\text{\textbf{Moment}}(F)$ denote the measure on $X_0$, whose value on $G$ is
   \[
   |\# Sur(G,F)| \mu_0(G).
   \]  
 \end{definition}

We begin by relating moment measures to certain probability distributions coming from random matrices.

 \begin{mysection}
\begin{lemma}
\label{lem:momentindependence}
The measures \mo(F) are linearly independent.
\end{lemma}

\begin{proof}
First, we observe that the evaluation of $\mo[F]$ is zero on groups that do not admit a surjection to $F$. Now suppose that there was a linear relation
\begin{equation}
\label{eqn:momentsum}
\sum_{F} a_{F} \mo(F)=0
\end{equation}

Let $F^{min}$ be a group $F$ that appears that does not admit a surjection to any other group that appears in this sum. Evaluating (\ref{eqn:momentsum}) on $F^{min}$ shows that the coefficient of $F^{min}$ must be 0. This gives a contradiction.
\end{proof}

\end{mysection}

\subsection{Cokernels of some random matrices}

\begin{lemma}
\label{matrixlemma}
Suppose $M$ is an $n \times n$ matrix. Then the distribution of the cokernel of the $(n+k) \times (n+k)$ matrix
$$
\begin{mat}{cccccc}
&&&\tempzero&\hdots & 0 \\
&M&&\tempvdots& \ddots & \vdots \\
&&&\tempzero& \hdots & 0 \\ \hline
*&\multicolumn{4}{c}{\hdots}&* \\
\vdots&\multicolumn{4}{c}{\ddots}&\vdots \\
*& \multicolumn{4}{c}{\hdots}&* \\
\end{mat}
$$
is
\begin{equation}
\label{fundamentalexpression}
d_{k}(coker(M) \times \Z_{p}^k)
\end{equation}
\end{lemma}

\begin{proof}
\begin{itemize}
\item First of all, the cokernel of 

$$
\begin{mat}{cccccc}
&&&\tempzero&\hdots & 0 \\
&M&&\tempvdots& \ddots & \vdots \\
&&&\tempzero& \hdots & 0 \\
\end{mat}
$$

is the trivial $\Z_p^k$-extension of $B$, i.e. $B \times \Z_p^k$.
\item Secondly, the cokernel of 

$$
\begin{mat}{cccccc}
&&&\tempzero&\hdots & 0 \\
&M&&\tempvdots& \ddots & \vdots \\
&&&\tempzero& \hdots & 0 \\ \hline
*&\multicolumn{4}{c}{\hdots}&* \\
\vdots&\multicolumn{4}{c}{\ddots}&\vdots \\
*& \multicolumn{4}{c}{\hdots}&* \\
\end{mat}
$$

is the quotient of $B \times Z_p^k$ by a uniformly random homomorphism from $\Z_p^k$, i.e. $\dk(B \times \Z_p^k)$
\end{itemize}
\end{proof}

\subsection{Limit formula for $d_{k}(B \times \Z_p^{k})$ in terms of moments.}
We relate $d_{k}(B \times \Z_p^{k})$ to moments.
\begin{theorem}
\label{limitlemma}
We have the following limit:
$$\lim_{k \rightarrow \infty} d_{k}(B \times \Z_p^{k}) = \mo(B)$$

This limit holds
\begin{itemize}
 \item pointwise,
 \item in $L^1(X_0)$ ,
 \item in $\ltwostar$.
 \end{itemize}
\end{theorem}

The rest of this section will be devoted to the proof of \autoref{limitlemma}.
\autoref{limitlemma} will be used twice: in the proof of \autoref{relationfrompresentation} and \autoref{thm:leadingtermone:mainbody}.

For the definition of the $L^1(X_0)$ norm, see \S \ref{sec:lone} in the appendix to this chapter.

\subsubsection{Point-wise limit}

The main tool will be the relation
$$(\dk)^T=\dkstar$$
proven earlier in \autoref{kgeneralselfadjtheorem}. We use the explicit form:

\begin{equation}
\label{generaladjointnessrelation}
\mu_k(B \times \Z_p^{k}) \probP(B \times \Z_p^{k} \xrightarrow{\dk} G)=\mu_0(G) \probP(G \xrightarrow{\dkstar} B \times \Z_p^{k} )
\end{equation}
In this expression, 
$$
\probP(B \times \Z_p^{k} \xrightarrow{\dk} G)
$$
is the expression we are interested in. It is precisely the evaluation of the measure \autoref{fundamentalexpression} at $G$.

\begin{sublemma}
$$
\probP \Big( G \xrightarrow{\dkstar} B \times \Z_p^{k} \Big)=\probP \Big( ker(G\rightarrow (\Q_p / \Z_p)^k ) \cong B \Big)
$$
where the homomorphism
$$
G\rightarrow \left( \ssfrac{\Q_p}{\Z_p} \right)^k 
$$
is chosen to be Haar random.
\end{sublemma}
This follows from the relation between characters and extensions in \S  \ref{sec:charactersandextensions}.


\marginparr{Check this reference - does not entirely follow}

Granting this, we calculate:

$$\probP \Big( ker(G\rightarrow (\Q_p / \Z_p)^k ) \cong B \Big)$$
To every $G\rightarrow (\Q_p / \Z_p)^k$ whose kernel is isomorphic to $B$, there is obviously an associated injection.
$$
B \hookrightarrow G \in Inj(B,G)
$$
determined up to an element of $Aut(B)$. We can think of this as associating to $G\rightarrow (\Q_p / \Z_p)^k$ an $Aut(B)$-orbit in $Inj(B,G)$, or alternatively a \textit{random choice of an element} in this orbit, where each of these equivalent possibilities is given weight 
\[
\frac{1}{|Aut(B)|}.
\] 


Thus, we can write: 
$$\probP \Big( ker(G \rightarrow ( \Q_p / \Z_p)^k ) \cong B \Big)=$$
$$ 
=\frac{1}{|G|^k} \sum_{\alpha \in Inj(B,G)} \frac{\Big|Inj \Big(\GB, (\Q_p / \Z_p)^k \Big)\Big|}{|Aut(B)|}
$$
where $\GB$ denotes the quotient of $G$ by the image of $\alpha$.
We rewrite this expression in the form
$$ 
\frac{1}{|Aut(B)||B|^k} \sum_{\alpha \in Inj(B,G)} \frac{\Big|Inj \Big(\GB, (\Q_p / \Z_p)^k \Big)\Big|}{|\GB|^k}.
$$
Substituting the other terms into 
\autoref{generaladjointnessrelation} now gives us the equality
$$\mu_k(B \times \Z_p^{k}) \probP(B \times \Z_p^{k} \xrightarrow{\dk} G)=$$
$$=\mu_0(G) \frac{1}{|Aut(B)||B|^k} \sum_{\alpha \in Inj(B,G)}\frac{|Inj \Big(\GB, (\Q_p / \Z_p)^k \Big)|}{|\GB|^k}$$

Recalling from \S \ref{sec:measures}, that  
$$\mu_k(B \times \Z_p^{k})=\frac{c_k}{|Aut(B)||B|^k}$$
where $c_k$ is a normalization constant, we obtain:
\begin{equation}
\label{limiteqmoment}
\probP(B \times \Z_p^{k} \xrightarrow{\dk} G)= \frac{\mu_0(G)}{c_k} \sum_{\alpha \in Inj(B,G)} \frac{|Inj \Big(\GB, (\Q_p / \Z_p)^k \Big)|}{|\GB|^k}.
\end{equation}

Finally, we can simplify this expression with the following two lemmas:

\begin{sublemma}
As $k \rightarrow \infty$,
$$c_k \rightarrow 1$$
\end{sublemma}

\begin{proof} 
This follows from the explicit from of $c_k$ in \S \ref{sec:measures}. 
\end{proof}

\begin{mymysection}
\begin{remark} There is moreover a natural way to see this: from its definition, the coefficient $c_k$ is the probability that 
$$
\dkstar \mu_0
$$
has no torsion as $k \rightarrow \infty$. This latter statement is an easy consequence of \autoref{thm:charactersandextensions} and \autoref{lem:charactersandextensions:isomorphism}.
\end{remark}
\end{mymysection}
\marginparr{Maybe Improve Formulation}

\begin{sublemma}
\label{lem:injequation}
Let $F$ be any finite group. Then,
$$\frac{\Big|Inj \Big(F, (\Q_p / \Z_p)^k \Big)\Big|}{|F|^k} \rightarrow 1$$
\end{sublemma}
\begin{proof}
We can rewrite the left hand side as:
\[
\frac{\Big|Inj \Big(F, (\Q_p / \Z_p)^k \Big)\Big|}{\Big|Hom \Big(F, (\Q_p / \Z_p)^k \Big)\Big|}
\]is the probability that a uniformly random homomorphism from $F$ to $(\Q_p / \Z_p)^k$ is injective. This probability tends to $1$ monotonically as $k \rightarrow \infty$. 
\end{proof}

\begin{mymysection}
\begin{sublemma}
\label{lem:injequation}
Let $\alpha$ be a fixed injection $B \hookrightarrow G$. Then,
$$\frac{\Big|Inj \Big(\GB, (\Q_p / \Z_p)^k \Big)\Big|}{|\GB|^k} \rightarrow 1$$
\end{sublemma}
\begin{proof}
This expression is the probability that a uniformly random homomorphism from $\GB$ to $(\Q_p / \Z_p)^k$ is injective. This is $1$ as $k \rightarrow \infty$. 
\end{proof}
\end{mymysection}

Thus, we get 

\[
 \lim_{k \rightarrow \infty} \probP(B \times \Z_p^{k} \xrightarrow{\dk} G)= \mu_0(G) \sum_{\alpha \in Inj(B,G)} 1=
\]
\[
= \mu_0(G) |Inj(B,G)|=\mu_0(G) |Sur(G,B)|.
\]

\begin{corollary}
We get another proof of the well-known fact that 
\begin{equation}
\label{eqn:momentone}
\sum_{G} \mu_0(G) |Sur(G,B)|=1= \sum_{G} \mu_0(G) |Inj(B,G)|
\end{equation}
\end{corollary}
\begin{proof}
From the fact that 
$$
\dk (B \times \Z_p^{k})
$$
is a probability measure, we can conclude by (\ref{limiteqmoment}) that
$$
c_k = \sum_{G} \mu_0(G) \sum_{\alpha \in Inj(B,G)} \frac{|Inj(\GB, \Q_p / \Z_p)^k|}{|\GB|^k}
$$
Taking the limit $k\rightarrow \infty$ on both sides and using \autoref{lem:injequation}, we get
\[
1 = \sum_G \mu_0(G) \sum_{\alpha \in Inj(B,G)} 1 =  \sum_{G} \mu_0(G) |Inj(B,G)|
\]
\begin{remark} 
The interchange of the sum and the limit is justified by the fact that \[\frac{|Inj(\GB, \Q_p / \Z_p)^k|}{|\GB|^k}\] is monotonically increasing in $k$.
\end{remark}
\end{proof}

\subsubsection{$L^1$ limit}

\begin{theorem}
The limit in \autoref{limitlemma} holds in $L^1(X_0)$, i.e.
\[
\lim_{k \rightarrow \infty}\sum_G \Big| \probP(B \times \Z_p^{k} \xrightarrow{\dk} G)\, - \, \mu_0(G) |Inj(B,G)| \Big|=0
\]
\end{theorem}
\begin{proof}

First we note that 
\[
\Big| c_k \probP(B \times \Z_p^{k} \xrightarrow{\dk} G) - \probP(B \times \Z_p^{k} \xrightarrow{\dk} G)\Big| \leq |c_k-1| \rightarrow  0
\] 
Hence it is sufficient to show:
\[
c_k \probP(B \times \Z_p^{k} \xrightarrow{\dk} G) \rightarrow \mu_0(G) |Inj(B,G)|
\]
in $L^1(X_0)$. By (\ref{limiteqmoment}),
\[
c_k \probP(B \times \Z_p^{k} \xrightarrow{\dk} G) - \mu_0(G) |Inj(B,G)| = \mu_0(G) \sum_{\alpha \in Inj(B,G)} \left(\frac{|Inj \Big(\GB, (\Q_p / \Z_p)^k \Big)|}{|\GB|^k} - 1 \right).
\]
Now, we apply the dominated convergence theorem and use the fact that
$$0 \leq  1-\frac{|Inj \Big(\GB, (\Q_p / \Z_p)^k \Big)|}{|\GB|^k} \leq 1$$
\end{proof}

\subsubsection{$\ltwostar$ limit}

\begin{lemma}
Convergence holds in $\ltwostar$.
\end{lemma}
Proceeding as in the proof of $L^1$ convergence, we show

\[
\mu_0(G) \sum_{\alpha \in Inj(B,G)} \frac{|Inj \Big(\GB, (\Q_p / \Z_p)^k \Big)|}{|\GB|^k}\rightarrow \mu_0(G) |Inj(B,G)|
\]
 in $\ltwostar$, or that 
$$\sum_{G} \mu_0(G) \left( \sum_{\alpha \in Inj(B,G)} \left(1-\frac{|Inj \Big(\GB, (\Q_p / \Z_p)^k \Big)|}{|\GB|^k} \right) \right)^2$$
tends to 0.

This expression under the first summation sign is dominated by 

$$\mu_0(G) |Inj(B,G)|^2$$

Therefore, we can conclude the result from the dominated convergence theorem, if we can show that 

\begin{equation}
\label{l2bound}
 \sum_{G} \mu_0(G) |Inj(B,G)|^2 < \infty
\end{equation}

\begin{proof}
But 
$$
\sum_{G} \frac{|Inj(B,G)|^2}{|Aut(G)|} \leq \sum_{G} \frac{|Hom(B,G)|^2}{|Aut(G)|} \leq \sum_{G} \frac{|Hom(B^2,G)|}{|Aut(G)|} \leq 
$$

$$
\leq \sum_{F\text{ a quotient of }B^2} \sum_{G} \frac{|Inj(F,G)|}{|Aut(G)|}=\sum_{F\text{ a quotient of }B^2} 1 \leq \infty
$$
\end{proof}

\subsection{The fundamental relation}

\begin{definition}
We define a random operation on groups as follows. Let $\del_0(G)$ be the kernel of a uniformly random element of $Hom(G, \Q_p/\Z_p)$. We extend $\del_0$ to an operator from probability distributions on $X_0$ to probability distributions on $X_0$, by linearity.
\end{definition}

\begin{mysection}
\begin{remark}
Because a finite abelian group is isomorphic to its dual, we could have equivalently defined $\del_0(G)$ to be the quotient of $G$ by a uniformly random element of $Hom(\Z_p,G)$. However the definition above agrees well with some other choices we make, in particular in \autoref{extcharlemma}.
\end{remark}
\end{mysection}

\begin{lemma}[The main relation]
\label{actionofdelta}
$$
\Delta_0(d^k(G \times \Z_p^k))=d^{k+1}(  \Z_p^{k+1} \times \del_0 G)
$$
Second version:
\[
\Delta_0(d^k(G \times \Z_p^k))=d^{k+1}( \Z_p^{k} \times \dstar G)
\]
\end{lemma}

\begin{proof}(of \autoref{actionofdelta})

We are interested in the cokernel of the $n+k+1 \times n+k+1$ random matrix:

$$
\begin{mat}{ccccccc}
 &&&\tempzero&\hdots & 0 & \tempstar \\
&M&&\tempvdots& \ddots & \vdots & \tempvdots \\
&&&\tempzero& \hdots & 0 & \tempstar \\ \hline
*&\multicolumn{4}{c}{\hdots}&* & \tempstar \\
\vdots&\multicolumn{4}{c}{\ddots}&\vdots & \tempvdots \\
*& \multicolumn{4}{c}{\hdots}&* & \tempstar\\  \hline
*& \multicolumn{4}{c}{\hdots}&* & \tempstar
\end{mat}
$$

Evidently, this is the same as the cokernel of the $n+k+1 \times n+k+1$ random matrix:

$$
\begin{mat}{ccccccc}
&&& \tempstar &\tempzero&\hdots & 0  \\
&M&& \tempvdots&\tempvdots& \ddots & \vdots  \\
&&& \tempstar&\tempzero& \hdots & 0  \\ \hline
*&\multicolumn{5}{c}{\hdots}&*  \\
\vdots&\multicolumn{5}{c}{\ddots}&\vdots  \\
*& \multicolumn{5}{c}{\hdots}&* \\
*& \multicolumn{5}{c}{\hdots}&* 
\end{mat}
$$

We proceed to understand the cokernel of this matrix by understanding the cokernels of a sequence of its minors: 
\begin{enumerate}
\item The cokernel of

$$
\begin{mat}{ccccccc}
&&& \tempstar \\
&M&& \tempvdots \\
&&& \tempstar \\ 
\end{mat}
$$

is 
$$d^{*}(B)$$
\item The cokernel of

$$
\begin{mat}{ccccccc}
&&& \tempstar &\tempzero&\hdots & 0  \\
&M&& \tempvdots&\tempvdots& \ddots & \vdots  \\
&&& \tempstar&\tempzero& \hdots & 0 
\end{mat}
$$

is 
$$d^{*}(B) \times \Z_p^k.$$

\item Finally, the cokernel of

$$
\begin{mat}{ccccccc}
&&& \tempstar &\tempzero&\hdots & 0  \\
&M&& \tempvdots&\tempvdots& \ddots & \vdots  \\
&&& \tempstar&\tempzero& \hdots & 0  \\ \hline
*&\multicolumn{5}{c}{\hdots}&*  \\
\vdots&\multicolumn{5}{c}{\ddots}&\vdots  \\
*& \multicolumn{5}{c}{\hdots}&* \\
*& \multicolumn{5}{c}{\hdots}&* 
\end{mat}
$$

is the quotient of $d^{*}(B) \times \Z_p^k$ by a random homomorphism from $\Z_p^{k+1}$, or
$$\dkplusone(d^{*}(B) \times \Z_p^k)$$
\end{enumerate}
Thus to prove \autoref{actionofdelta}, it remains to show
\begin{sublemma}
\label{extcharlemma}
$$d^{*}(B)= \del(B) \times \Z_p$$
\end{sublemma}

\marginparr{Extra Work? More intrinsic proof?}

This follows from the discussion on characters and extensions in \S \ref{sec2} and from \autoref{lem:charactersandextensions} \commm{Maybe improve}.
\end{proof}

\subsubsection{Action of $\Delta_0$ on moments}

Below, we will take the limit of the relation in \autoref{actionofdelta} to derive \autoref{relationfrompresentation}, which will play an important role in the sequel.
\marginparr{Be more careful with this}

\begin{definition}
We define the operator $\mo$ on \textit{signed finitely supported measures} on groups as follows:
$$ \mo(G) $$ is the same distribution as previously defined, and we extend to signed finitely supported measures by linearity.
\end{definition}

\begin{remark}
In fact, $\mo$ can be extended to an operator from $L^1(X_0)$ to $L^1(X_0)$, but this will not be necessary for the sequel.
\end{remark}

\begin{mysection}
\begin{remark}
It is well-known, and follows from \autoref{limitlemma}, that for every $G$, $\mo(G)$ is a probability measure. Hence, in the above definition, we can replace the condition of finitely supported by something more general. However, in all the cases we consider, we will not need anything more than in the definition above.
\end{remark}
\end{mysection}

We can take the limit of \autoref{actionofdelta} to get the following important relation:

\begin{theorem}
\label{relationfrompresentation}
The relation
\[
\Delta_0 \Big( \mo(\emptydot) \Big)=\mo \Big( \del(\emptydot) \Big)
\]
holds for all finitely supported measures ($\emptydot$).
\end{theorem}

\begin{proof}
We take the limit of the expression in \autoref{actionofdelta} as $k\rightarrow \infty$. We have convergence: 
\[
\lim_{k \rightarrow \infty} d_{k}(B \times \Z_p^{k}) = \mo(B)
\]
in $\lone$ and $\Delta_0$ is a bounded operator $\lone \rightarrow \lone$. Hence 
\[
\Delta_0 \Big( \mo(\emptydot) \Big)=\mo \Big( \del(\emptydot) \Big)
\]
follows.
\end{proof}

\marginparr{The following section needs to be edited}
\subsection{Eigenfunctions of $\del_0$}

\begin{definition}
We write $G' \leq G$ if there is a surjective map from $G$ to $F$. Furthermore, we write 
$F < G$ when $F \leq G$ and $F$ is not isomorphic to $G$.
\end{definition}

This is a partial order relation on $X_0$.

\begin{lemma}
\label{lem:delistriangular}
Under the above ordering on groups, the operator $\del_0$ is upper-triangular. Furthermore the diagonal entry on the row indexed by $F$, is $\frac{1}{|F|}$.
\end{lemma}

\begin{proof}
This follows immediately from the definition of $\del_0$.
\end{proof}
\marginparr{Maybe give more detail in the proof of this later.}

\marginparr{Possibly improve formulation of last proposition and next paragraph}

Upper diagonal operators are easy to diagonalize. First, we introduce a notation for singleton measures:
\begin{definition}
We write $\one_G$ to denote a probability measure on $X_0$, all of whose mass is concentrated on $G$.
\end{definition}

From \autoref{lem:delistriangular}, we can immediately conclude the following:
\begin{lemma}
\label{lem:explicitdel}
\begin{enumerate}
\item For every $F \in X_0$, $\del_0$ there is an eigenfunction of $\del_0$ the form:
$$ \one_{F}+\sum_{F'<F} a_{F'} \one_{F'}$$ 
We will denote this eigenfunction as $e_F$.
\item The eigenvalue associated to the eigenfunction $e_F$ is $\frac{1}{|F|}$
\item These $e_F$ span all functions with finite support.
\end{enumerate}
\end{lemma}

\begin{proof}
This is a straightforward consequence of \autoref{lem:delistriangular}.
\end{proof}

\begin{mysection}
Below, we tabulate the $e_F$ explicitly for some small groups.
\end{mysection}

\begin{mysection}
\begin{remark} 
\label{explicitdel}
In particular, it follows that for every finitely supported measure $\nu$ on $\Xzero$, we can express
$$
\del ^ N (\nu)
$$
as a finite linear combination of terms of the form
$$
\frac{1}{|F|^N}e_F
$$
\end{remark} 
\end{mysection}

\begin{remark} 
\label{explicitdel}
In particular, it follows that for every finitely supported measure $\nu$ on $\Xzero$, we can express
$$
\del ^ N (\nu)
$$
as a finite sum
$$
\sum_F a_F \frac{1}{|F|^N}e_F
$$
for some coefficients $a_F$.
\end{remark}

\subsection{Construction of eigenfunctions of $\Delta_0$}

\begin{lemma}
\begin{equation}
\label{eqn:mo}
\Delta_0 \mo(e_F)= \frac{1}{|F|} \mo(e_F)
\end{equation}
i.e. $\mo(e_F)$ is an eigenfunction of $\Delta_0$ with eigenvalue $\frac{1}{|F|}$. We will denote $\mo(e_F)$ as $E_F$.
\end{lemma}

 \begin{proof}
 (\ref{eqn:mo}) follows from the expression in \autoref{relationfrompresentation}. 
 \end{proof}

\begin{corollary}
$E_F$ is of the form
\[
\mo(F) + \sum_{F'<F} a_{F'} \mo(F')
\]
for some coefficients $a_{F'}$. Consequently, we can express
\[
\mo(F)
\]
as a finite linear combination of the $E_F$.
\end{corollary}

\marginparr{(An idea is given in the comments. Take note.)}

\begin{mysection}

can a posteriori be seen that this can be done, though maybe it is not a priori obvious: but there should be a way to do this... reconstruct a function from 

its moments... so we find a candidate expression, using your explicit formula, and then we show it has the right moments, and hence by uniqueness, it must be 

the correct expression.
\end{mysection}

\subsection{Alternative description of eigenfunctions of $\Delta_0$}

\begin{definition}
$V^{\leq G}$ is the linear span of 
\[
\Big\{ \mo(F)\Big| F<G \Big\}
\]
 $V^{<G}$ is the linear span of
 \[
\Big\{ \mo(F)\Big| F \leq G \Big\}
 \]
\end{definition}

\begin{lemma}
\label{lem:generator}
$E_G$ is a generator for the orthogonal complement of $V^{<G}$ in $V^{\leq G}$.
\end{lemma}

First, we show the following:
\begin{sublemma}
\label{lem:span}
The finite-dimensional space $V^{<G}$ is spanned by $\Big\{ E_F \Big| F<G \Big\}$.
The finite-dimensional space $V^{\neq G}$ is spanned by $\Big\{ E_F \Big| F \neq G \Big\}$.
\end{sublemma}
\begin{proof}
This follows by dimension-counting.
\end{proof}

$E_G$ has eigenvalue $\frac{1}{|G|}$. By self-adjointness of $\Delta_0$, it is orthogonal to  $E_F$ if $F<G$. Hence by \autoref{lem:span}, $E_G$ is orthogonal to $V^{<G}$.

\begin{proof}(Of \autoref{lem:generator})
The orthogonal complement of $V^{<G}$ in $V^{\leq G}$ is $1$-dimensional and contains $E_G$. Hence, $E_G$ is a generator for the orthogonal complement of $V^{<G}$ in $V^{\leq G}$.
\end{proof}

\begin{mysection}
\subsection{Alternative, alternative derivation of formula}

\begin{lemma}
The action of $\Delta_0$ on the basis $\fouriermap(1_F)$ is upper-diagonal. Moreover, the diagonal entries are 
\[
|F|^{-1}
\]
\end{lemma}

\begin{lemma}
The functions $\fouriermap(1_F)$ are mutually orthogonal.
\end{lemma}
\end{mysection}

\subsection{Preliminary version of the second main theorem}
\label{sec:alleigenfunctions}

\begin{lemma}
The eigenfunctions $E_{F}$ lie in $\ltwostar$.
\end{lemma}

\begin{proof}

It is sufficient to show that, for a group $F$, the function

$$
\mo(F)
$$
lies in $\ltwostar$, or that 
\begin{equation}
\label{eqn:momentltwosum}
\sum_{G} \frac{|Sur(G,F)|^2}{|Aut(G)} < \infty
\end{equation}
But this follows from \autoref{l2bound}.

\end{proof}

\marginparr{Maybe erase. OPTIONAL. Comment hidden in remark.}
\begin{mysection}
\begin{remark}
Here, it suffices to give an upper down for the norm. Further, in \autoref{thm:curioussum}, we give an exact expression for the sum (\ref{eqn:momentltwosum}). 
\end{remark}
\end{mysection}

The next theorem shows that, if $\nu$ is a finitely supported measure, the leading term of 
\[
\Delta_0^N (\nu)
\]
is always a finite linear combination of the $E_F$'s:

\begin{theorem}
\label{thm:leadingtermone:mainbody}
For any finitely supported measure $\nu$, there exist $\llambda \in \mathbb{N}$, $F_i \in X_0$ and coefficients $a_{F_i} \in \mathbb{R}$ such that 
\[
\llambda = |F_i| \, \forall \, i
\]
and
\[
\Delta_0^N \nu = \frac{1}{\llambda^N} \sum_i a_{F_i} E_{F_i} + o\left(\frac{1}{\llambda^N}\right) 
\]
\end{theorem}

\begin{proof}

By the remark following \autoref{lem:explicitdel},
$$\del^N(\nu)=\sum_{F} a_F \frac{e_F}{|F|^N}$$

where the sum is finite. Hence,

$$
\Delta_0^N (\nu) = \sum_{F} a_F \frac{1}{|F|^N} d^N ( e_F \times \Z_p^N)
$$

By \autoref{limitlemma}, this is 
\begin{equation}
\label{eqn:littleosum}
\sum_{F} a_F \frac{1}{|F|^N} \Big(\mo( e_F ) + o(1)\Big)
\end{equation}
where the limit, implicit in the "little o" notation, is taken in the $\ltwostar$ topology. Now define
\[
\lambda \defeq \min_{F} \Big\{ |F| \Big| a_F \neq 0 \Big\}
\]
The expression (\ref{eqn:littleosum}) can be re-written as:
\[
\sum_{\llambda=|F'|} a_{F'} \frac{1}{\llambda^N}\mo(e_F') + o\left(\frac{1}{\llambda^N}\right)=
\frac{1}{\llambda^N}\sum_{\llambda=|F'|} a_{F'} E_{F'} + o\left(\frac{1}{\llambda^N}\right)
\]
\end{proof}

\begin{remark}
With the notation as above, the inner product 
\begin{equation}
\label{eqn:orthogonality}
\Big< \sum_i a_{F_i} E_{F_i} \Big| \nu \Big> \neq 0
\end{equation}
\end{remark}

\begin{proof}
We use self-adjointness:
\[
\Big< \sum_i a_{F_i} E_{F_i} \Big| \Delta_0^N\nu \Big> =
\Big< \Delta_0^N \sum_i a_{F_i} E_{F_i} \Big| \nu \Big>
\]
The result follows by taking the limit $N \rightarrow \infty$ and using the \autoref{thm:leadingtermone:mainbody}.
\end{proof}

\begin{definition}
Let $\hil_\lambda$ be the orthogonal complement of 
\[
\Big\{ E_F \Big| F \leq \lambda \Big\}
\]
\end{definition}

\begin{lemma}
\label{lem:lambdanorm}
 The restriction of $\Delta_0$ to $\hil_\lambda$ has norm at most $\lambda^{-1}$.
\end{lemma}

\begin{proof}
Suppose that $\nu$ is a \textit{finitely supported measure} in $\hil_\lambda$. Then, it follows from \autoref{thm:leadingtermone:mainbody} and the ensuing remark that:
\begin{equation}
\label{eqn:lambdalimit}
\lim_{N \rightarrow \infty} \lambda^N \Delta_0^N \nu = 0
\end{equation}
Now, given (\ref{eqn:lambdalimit}), the result follows from the Spectral Theorem. The relevant statements are in the appendix:
\begin{itemize}
\item[1)] By \autoref{lem:density} in the appendix, finitely supported measures are dense in $\hil_\lambda$. 
\item[1)] Therefore, it follows from \autoref{lem:hilbert:appendix} in the appendix that
\[
\Delta_0 \Big|_{\hil_\lambda}
\] 
has norm at most $\lambda^{-1}$.
\end{itemize}
\end{proof}

\begin{definition}
Define:
\[
\hil_{\infty} \defeq \bigcap_{\lambda} \hil_{\lambda}
\]
\end{definition}

\begin{theorem}[Second main theorem, preliminary version]
\label{thm:secondmainprothm}
\[
\hil_{\infty}=ker(\Delta_0) \, \cap \, \ltwostar
\]
\begin{mysection}
The orthogonal complement of $im(\fouriermap)$ is $ker(\Delta_0)$.
\end{mysection}
\end{theorem}

\begin{proof}
\begin{mysection}
The orthogonal complement of $im(\fouriermap)$ is the orthogonal complement of all the $E_F$. Hence the orthogonal complement of $im(\fouriermap)$ is 
\[
\bigcap \, \hil_{\lambda}
\]
\end{mysection}
By \autoref{lem:lambdanorm}, the restriction of $\Delta_0$ to
\[
\bigcap_{\lambda} \, \hil_{\lambda}
\] has norm at most
\[
\inf \{ \lambda^{-1} \}=0.
\] Hence
\[
\hil_{\lambda}=ker(\Delta_0) \, \cap \, \ltwostar
\]
\end{proof}

\begin{mymysection}
\section{Possibly unnecessary}
\begin{theorem}
Any eigenfunction of $\Delta_0$, that is in $\ltwostar$, is a linear combination of the $E_F$'s
\end{theorem}
... needs modification ... has non-zero eigenvalue ...

\begin{proof}
Suppose the contrary, i.e. that there exists some eigenfunction $\Elambda$, with eigenvalue $\lambda$, that lies in $\ltwostar$, but is not a linear 

combination of the $E_F$'S. Choose a finitely supported measure $\nu$ such that 
\begin{enumerate}
\item $<\nu,E_F>_{CL}=0$ whenever $\frac{1}{|F|} \geq \lambda$.

\end{enumerate}

The first condition implies that

$$
<\Delta_0^n \nu,E_F>_{CL}=0
$$
whenever $\frac{1}{|F|} \geq \lambda$.
By \autoref{leadingterm}, we have 
$$
\lim_{n \rightarrow \infty}(p^{k})^N \delta_0^N (\nu) =  \sum_{|F|=p^k} a_F  E_{F}.
$$ 
in $\ltwostar$.
This implies that $p^{-k} < \lambda$
Hence,
$$
\lim_{n \rightarrow \infty}\frac{1}{\lambda^N} \delta_0^N (\nu) =  0
$$
in $\ltwostar$.
Hence,
$$
\lim_{n \rightarrow \infty} \frac{1}{\lambda^N} <\Delta_0^N \nu,\Elambda>=0
$$

But
$$
\frac{1}{\lambda^N} <\delta_0^N \nu,\Elambda>=\frac{1}{\lambda^N} <\nu , \delta_0^N \Elambda>=<\nu , \Elambda>
$$
We summarize. Whenever a finitely supported measure $\nu$ is orthogonal to finitely many of $E_F$, then it is also orthogonal to $\Elambda$. Hence the orthogonal complement of $\Elambda$ contains the orthogonal complement of finitely many of the $E_F$, and therefore $\Elambda$ must lie in the span of finitely many $E_F$.
\end{proof}

\marginparr{Idea given in notes. Take note.}
\begin{mysection}

to take out certain modes... Then some modes appear in the main expression... but we do not have exact modes, just stuff that converges to them...


particular get a good description of the leading term using this approach.
\end{mysection}

\end{mymysection}

\subsection{Proofs of the Main Theorems}
\subsubsection{Curious Formula}

\begin{theorem}
\label{thm:curioussum:mainbody}
For any $F_1, F_2 \in X_0$, we have 
\begin{equation}
\label{eqn:curioussum:mainbody}
c_0\sum_G \frac{\#Sur( G, F_1 )\#Sur( G, F_2 )}{\#Aut(G)}=\sum_{G'}\frac{|\#Sur(F_1,G')||\#Sur(F_2,G')|}{|\#Aut(G')|}
\end{equation}
\end{theorem}

\begin{remark}
Taking $F_1$ to be the trivial group recovers the well-known fact that all the moments of the Cohen-Lenstra measure are $1$.
\end{remark}

We begin the proof by noting that the left hand side of (\ref{eqn:curioussum:mainbody}) is a weighted sum over diagrams 
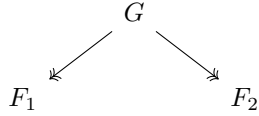
\[
\begin{tikzcd}
&G \arrow[dl, \sur] \arrow[dr, \sur]&\\
F_1&&F_2\\
\end{tikzcd}
\]

\begin{definition}
Let $Hom'(G,F_1 \times F_2)$ denote the set of homomorphisms from $G$ to $F_1 \times F_2$ that are \textit{surjective} to both factors.
\end{definition}

Then,
\[
c_0\sum_G \frac{\#Sur( G, F_1 )\#Sur( G, F_2 )}{\#Aut(G)} = \]\[
c_0\sum_G \frac{\#Hom'(G, F_1 \times F_2)}{\#Aut(G)} =\]\[
c_0\sum_{G,K} \frac{\#Sur(G,K)\#Inj'(K, F_1 \times F_2)}{\#Aut(G)\#Aut(K)}
\]

where $\#Inj'(K, F_1 \times F_2)$ denotes the number of injections from $K$ to $F_1 \times F_2$ that are surjective to both factors. Now, by (\ref{eqn:momentone}),
\[
c_0\sum_{G} \frac{\#Sur(G,K)}{\#Aut(G)}=1
\]

Hence, the above sum becomes
\begin{equation}
\label{eqn:numberofsubgroups}
\sum_{K} \frac{\#Inj'(K, F_1 \times F_2)}{\#Aut(K)}
\end{equation}

This is the number of subgroups of $F_1 \times F_2$ that project surjectively to both factors. 

\begin{lemma}
The number of subgroups of $F_1 \times F_2$ that project surjectively to both factors is:
\[
\sum_{G'} \frac{\#Sur(F_1,G')\#Sur(F_2,G')}{\#Aut(G')}
\]
\end{lemma}

\begin{lemma}
Every such subgroup $K$ fits into a Cartesian square:
\begin{diagram}
& K \arrow[dr,\sur, "\pi_2"] \arrow[dl,\sur, "\pi_1"'] & \\
F_1 \arrow[dr,\sur] &  & F_2 \arrow[dl,\sur]\\
& G' & \\
\end{diagram}
\end{lemma}

\begin{proof}
Take 
\[
G' \cong K \Big/ \Big(ker(\pi_1) + ker(\pi_1) \Big)
\]
\end{proof}

Hence, the problem reduces to counting Cartesian diagrams. A Cartesian diagram is completely determined by the pair of homorphisms:
\begin{equation}
\label{eqn:bottomofcartesian}
\begin{tikzcd}
F_1 \arrow[dr, \sur] &  & F_2 \arrow[dl, \sur]\\
& G' & \\
\end{tikzcd}
\end{equation}
Two distinct diagrams (\ref{eqn:bottomofcartesian}) yield the same subgroup of $F_1 \times F_2$ if and only if they are related by an automorphism of $G'$.  This proves the lemma.

\subsubsection{Construction of $\fouriermap$}

First, we define the map that takes finite linear combinations of functions $\frac{\# Sur (F,\emptydot)}{\# Aut(\et)}$, to finite linear combinations of functions $\frac{\# Sur (\emptydot,F')}{\# Aut(\et)}$ by sending 
\[
\frac{\# Sur (F,\emptydot)}{\# Aut(\et)} \rightarrow \sqrt{c_0}\frac{\# Sur (\emptydot,F')}{\# Aut(\et)}
\]

\begin{lemma}
This map is well-defined.
\end{lemma}

\begin{proof}
This is true because the functions $\frac{\# Sur (F,\emptydot)}{\# Aut(\et)}$ are linearly independent.
\end{proof}

This map is norm-preserving by the identity in \autoref{thm:curioussum:mainbody}. 

\begin{lemma}
$\fouriermap$ extends to a norm-preserving homomorphism from $\ltwostar$ to the closure of the linear span of
\[
\frac{\# Sur (\emptydot,F')}{\# Aut(\et)}
\]
\end{lemma}

\begin{proof}
$\fouriermap$ extends to a norm-preserving homomorphism from the closure of the linear span of
\[
\frac{\# Sur (F,\emptydot)}{\# Aut(\et)}
\] 
 to the closure of the linear span of
\[
\frac{\# Sur (\emptydot,F')}{\# Aut(\et)}
\]
But every finitely supported function can be expressed as a linear combination of functions of the form $
\frac{\# Sur (F,\emptydot)}{\# Aut(\et)}
$ Finitely supported functions are dense in $L^2(X_0,\mu)$. Hence, the closure of the linear span of 
$
\frac{\# Sur (F,\emptydot)}{\# Aut(\et)}
$ is $L^2(X_0,\mu)$
\end{proof}

\subsubsection{A proof of the first main theorem}

We begin with the following lemma:
\begin{lemma}
$1_G$ is in the orthogonal complement of
\[
span \Big\{ \frac{\#Sur(F,\emptydot)}{\#Aut(\et)} \Big| F < G \Big\}
\]
in
\[
span \Big\{ \frac{\#Sur(F,\emptydot)}{\#Aut(\et)} \Big| F \leq G \Big\}
\]
\end{lemma}

\begin{proof}
$1_G$ is orthogonal to all $\frac{\#Sur(F,\emptydot)}{\#Aut(\et)}$, for $F < G$. Hence, all that needs to be shown is:
\[
1_G \in span\Big\{ \frac{\#Sur(F,\emptydot)}{\#Aut(\et)} \Big| F \leq G \Big\}
\]
This can be shown by induction on $G$. The details are left to the reader.
\end{proof}

Now we apply $\fouriermap$ to get
\begin{lemma}
$\fouriermap(1_G)$ is in the orthogonal complement of
\[
span \left\{ \frac{\#Sur(\emptydot,F)}{\#Aut(\et)} \Big| F < G \right\} 
\]
in
\[
span \left\{ \frac{\#Sur(\emptydot,F)}{\#Aut(\et)} \Big| F \leq G \right\} 
\]
\end{lemma}

Hence $\fouriermap(1_G)$ is in the orthogonal complement of
\[
span \Big\{ \mo[F] \Big| F < G \Big\} = V^{<G}
\]
in
\[
span \Big\{ \mo[F] \Big| F \leq G \Big\} = V^{\leq G}
\]

Therefore, $\fouriermap(1_G)$ is a multiple of $E_G$. Hence:
\[
\Delta_0 \fouriermap(1_G)  = |\#G|^{-1} \fouriermap(1_G)
\]
Or alternatively,
\[
\Delta_0 \fouriermap(1_G)  = \fouriermap\left(|\#G|^{-1} 1_G\right)
\]

By linearity, we get:

\begin{theorem*}[The First Main Theorem]
\[
\Delta_0 \fouriermap\left( \nu \right) = \fouriermap\left( |\#G|^{-1} \nu \right)
\]
\end{theorem*}

\begin{mysection}

\subsubsection{Passing to $\ltwoCL$}

\begin{lemma}
\[
\fouriermap(1_G)
\]
lies in 
\[
V^{\leq G}
\]
\end{lemma}

\begin{proof}
$V^{\leq G}$ is the linear span of 
\[
\Big\{ \mo(F)\Big| F<G \Big\}
\]
Applying the isometry $\fouriermap$,  it is sufficient to show that 
\[
1_G
\]
is in the linear span of the functions
\[
|\#Sur(F \leq G)|
\]
This can be shown by induction on $G$..........................
\end{proof}

\begin{lemma}
\[
span \Big\{ \fouriermap(1_F')   \Big|   F'<F   \Big\}
\]
spans 
\[
V^{< G}\]
\end{lemma}

\begin{proof}
Applying the isometry $\fouriermap$, it is sufficient to show that 
\[
span \Big\{ |\#Sur(F,\cdot)\Big| F<G \Big\}
\]
is spanned by 
\[
span \Big\{ 1_F \Big| F<G \Big\}
\]
But this is evidently true, because 
\[
|\#Sur(F,\cdot)
\]
is supported on 
\[
\Big\{F' \Big| F' \leq F \Big\}
\]
\end{proof}

\begin{lemma}
\[
\fouriermap(1_F)
\]
is orthogonal to 
\[
span \Big\{ \fouriermap(1_F')   \Big|   F'<F   \Big\}
\]
\end{lemma}

\begin{proof}
By the fact that $\fouriermap$ is an isometry, this follows from the fact that 
\[
\fouriermap(1_F)
\]
is orthogonal to 
\[
span \Big\{ \fouriermap(1_F')   \Big|   F'<F   \Big\}
\]
\end{proof}

\begin{theorem}
$\fouriermap(1_G)$ is a linear multiple of $E_G$.
\end{theorem}

\begin{proof}
By the above, $\fouriermap(1_G)$ lies inside the orthogonal complement of 
\[
V^{< G}
\]
in 
\[ 
V^{\leq G}
\] 
But by \autoref{...}, this is a one-dimensional space spanned by $E_G$.
\end{proof}

ADD MORE DETAIL

...main theorem...

\end{mysection}

\subsubsection{A proof of the second main theorem}

\begin{lemma}
\[
\hil_{\infty} = im(\fouriermap)^{\perp}
\]
\end{lemma}

\begin{proof}
This statement is nearly tautological. Indeed,
\[
\hil_{\infty}= span( E_F)^{\perp} 
\]
Hence,
\[
\hil_{\infty}= span(E_F)^{\perp} = span\Big(\fouriermap(1_F)\Big)^{\perp}
=
im(\fouriermap)^{\perp}
\]
\end{proof}

Hence, as a corollary, we get:
\begin{theorem*}[The Second Main Theorm]
\[
im(\fouriermap)^{\perp} = ker(\Delta_0) \cap \ltwostar 
\]
\end{theorem*}

\begin{proof}
Indeed, we know from \autoref{thm:secondmainprothm} that 
\[
\hil_{\infty}= ker(\Delta_0) \cap \ltwostar
\]
Hence
\[
\hil_{\infty}= ker(\Delta_0) \cap \ltwostar = ker(\Delta_0) \cap \ltwostar
\]
\end{proof}

\begin{mysection}
.....Elaborate on $\Delta_0$.....
\end{mysection}

\begin{mymysection}
\subsection{Possibly unnecessary}
\begin{theorem}
Any eigenfunction of $\Delta_0$, that is in $\ltwostar$, is a linear combination of the $E_F$'s
\end{theorem}
... needs modification ... has non-zero eigenvalue ...

\begin{proof}
Suppose the contrary, i.e. that there exists some eigenfunction $\Elambda$, with eigenvalue $\lambda$, that lies in $\ltwoCL$, but is not a linear 

combination of the $E_F$'S. Choose a finitely supported measure $\nu$ such that 
\begin{enumerate}
\item $<\nu,E_F>_{CL}=0$ whenever $\frac{1}{|F|} \geq \lambda$.

\end{enumerate}

The first condition implies that

$$
<\Delta_0^n \nu,E_F>_{CL}=0
$$
whenever $\frac{1}{|F|} \geq \lambda$.
By \autoref{leadingterm}, we have 
$$
\lim_{n \rightarrow \infty}(p^{k})^N \delta_0^N (\nu) =  \sum_{|F|=p^k} a_F  E_{F}.
$$ 
in $\ltwostar$.
This implies that $p^{-k} < \lambda$
Hence,
$$
\lim_{n \rightarrow \infty}\frac{1}{\lambda^N} \delta_0^N (\nu) =  0
$$
in $\ltwostar$.
Hence,
$$
\lim_{n \rightarrow \infty} \frac{1}{\lambda^N} <\Delta_0^N \nu,\Elambda>=0
$$

But
$$
\frac{1}{\lambda^N} <\delta_0^N \nu,\Elambda>=\frac{1}{\lambda^N} <\nu , \delta_0^N \Elambda>=<\nu , \Elambda>
$$
We summarize. Whenever a finitely supported measure $\nu$ is orthogonal to finitely many of $E_F$, then it is also orthogonal to $\Elambda$. Hence the orthogonal complement of $\Elambda$ contains the orthogonal complement of finitely many of the $E_F$, and therefore $\Elambda$ must lie in the span of finitely many $E_F$.
\end{proof}

\marginparr{Idea given in notes. Take note.}
\begin{mysection}

to take out certain modes... Then some modes appear in the main expression... but we do not have exact modes, just stuff that converges to them...


particular get a good description of the leading term using this approach.
\end{mysection}

\end{mymysection}

\subsection{Appendix}

Here, we group results which were used in this section.

\subsubsection{Definition of $L^1(X_0)$}
\label{sec:lone}
$L^1(X_0)$ is a space of measures on $X_0$. The norm of $\nu \in L^1(X_0)$ is defined as:
\[
||\nu|| \defeq \sum_{G \in X_0} |\nu(G)|
\]

\subsubsection{$\Delta_0$ is a bounded operator on $\ltwostar$}

\begin{mysection}
It is a general fact that a reversible Markov chain on a countable state space with stationary measure $\pi$ has norm $1$ as an operator on $L^2(\pi)$.
\end{mysection}

\begin{theorem}
\label{thm:boundedoperator}
$\Delta_0$ is a bounded operator on $\ltwostar$.
\end{theorem}

We use the fact that $\Delta_0$ can be represented as a random walk on a weighted countable graph.

\autoref{thm:boundedoperator} will follow from \autoref{lem:weightedgraphaverage} below.

\paragraph{The Random Walk Operator on a Weighted Graph has norm $1$.}

Suppose we have a graph with vertex set $V$ a weighted edge set $V \times V$ and probability measures $\rho: V \times V \rightarrow \R$ and $\mu: V \rightarrow \R$ such that 
\[
\rho(v,w) = \rho(w,v)
\]
\[
\sum_{w \in W} \rho(v,w) = \mu(v)
\]

\begin{definition}
Let $L^2(V,\mu)^{*}$ be the space of measures on $V$ with norm
\[
||\nu|| \defeq \sum_{w \in V}  \frac{\nu(w)^2}{\mu(w)}
\]
\end{definition}

\begin{lemma}
\label{lem:weightedgraphaverage}
The weighted random walk operator:
\[
\nu(\et) \mapsto \sum_{w \in V}\frac{\nu(w)\rho(\et,w)}{\mu(w)}
\]
is a bounded on $L^2(V,\mu)^{*}$.
\end{lemma}

\begin{proof}
\[
\sum_{v \in V} \frac{1}{\mu(v)} \left(\sum_{w \in V}\frac{\nu(w)\rho(v,w)}{\mu(w)}\right)^2 =
\]
\[
=\sum_{v \in V} \frac{\mu(v)}{\mu(w)^2} \left(\sum_{w \in V}\frac{\nu(w)\rho(v,w)}{\mu(v)}\right)^2\leq
\]
\[
\leq \sum_{v \in V} \frac{\mu(v)}{\mu(w)^2} \sum_{w \in V}\frac{\nu(w)^2\rho(v,w)}{\mu(v)}=
\]
\[
\sum_{v \in V} \sum_{w \in V} \frac{\rho(v,w)}{\mu(w)^2} \nu(w)^2= \sum_{w \in V} \frac{\nu(w)^2}{\mu(w)}
\]
\end{proof}

\begin{mysection}
Consider the space of functions on $X_0 \times X_0$. Equip $X_0 \times X_0$ with the measure induced by our weighted graph.

Let $\pi_i$ be the projection unto the $i^{th}$ factor. The pushforward of the weighted graph measure by $\pi_1$ is $\mu_0$.

We can defined projection and induction.

Induction is pull-back; this preserved the $L^2$ norm. Projection is conditional expectation
and decreases the $L^2$ norm. Averaging over neighbors is a composition of induction and projection and hence does not increase the norm. Averaging over neighbors preserves the constant functions. Hence the norm is $1$.
\end{mysection}

\marginparr{Find reference}

\begin{mysection}
\subsubsection{Proof}
\begin{proof}
The proof represents the Markov chain as a projection operator...
\end{proof}

We can model the action of ... on \textit{functions} as follows. First spread out over adjacent endpoints, then switch endpoints. Then average out over adjacent endpoints.

This is a sequence of isomorphisms and projections. Hence is a bounded operator.
\end{mysection}

\subsubsection{Input from Spectral Theory}

\begin{lemma}
\label{lem:hilbert:appendix}
Suppose that $T$ is a bounded self-adjoint operator $T$ on a Hilbert space $\hil$. Suppose further that $T$ satisfies:
\[
\lim_{N \rightarrow \infty} T^N v =0 
\]
for every $v$ lying in some dense subset of $\hil$. Then the operator norm of $T$ is at most $1$.
\end{lemma}

\begin{proof}
This is true for operators that act by multiplication. By the spectral theorem, \cite[IV.194]{BourbakiTS}, every bounded self-adjoint operator on a Hilbert space is conjugate, by an isometry, to an operator that acts by multiplication.
\end{proof}

\subsubsection{Finitely supported functions are dense in subspaces}

\begin{lemma}
\label{lem:density}
Let $W$ be a linear subspace of $L^2(X_0,\mu_0)$ of finite co-dimension. \textit{Finitely supported functions} are dense in $W$.
\end{lemma}

\begin{proof}
This is an exercise in linear algebra.
\end{proof}

\begin{mysection}
Let $W$ be a subset of the support such that $V^{\perp} \rightarrow W$ is surjective....
\end{mysection}

\begin{mymysection}

\section{The dual of $\Delta_0$}
Let $\Delta_0^{*}$ be the operator dual to $\Delta_0$. 
\begin{remark}
Explicitly, given a function $f$ on $X_0$,
\[
(\Delta_0^{*}f)(G) = \probE f(\Delta_0 \, G)
\]
\end{remark}
In other words, if we regard $\Delta_0$ as a \textit{random walk on a weighted graph}, $\Delta_0^{*}$ is a \textit{weighted average over neighbors}.

\section{Functions and measures}

In the introduction, we have phrased everything in terms of $\ltwoCL$.
In the ensuing sections it will be more natural to work with $\ltwostar$, the dual of $\ltwoCL$.

Of course, $\ltwoCL$ is a Hilbert space and hence isomorphic to its dual, via the pairing
\[
\Big< \emptydot , \, \emptydot \Big>
\]
\begin{mysection}
We will regard the pairing as giving an isomorphism:
\[
\ltwoCL \xrightarrow{\sim} \ltwostar
\]
that sends $f \in \ltwoCL$ to
\[
f(\emptydot)\mu_0(\emptydot)=c_0\frac{f(\emptydot)}{\#Aut(\emptydot)} 
\]
The inverse of this map sends $\ltwostar$ to 
\[
\ltwoCL \cong \ltwoCL^{**}
\]
and hence defines a symmetric pairing on $\ltwostar$, explicitly given by:
\[
\Big<\nu_1, \, \nu_2 \Big> = c_0^{-1} \sum_G \#Aut(G)|\nu_1(G)||\nu_2(G)| 
\]
Thus, $\ltwostar$ also has the structure of a Hilbert space. With this structure,
\[
\ltwoCL \rightarrow \ltwostar
\]
is an isometry.
\end{mysection}

\subsubsection{Generalities about $\ltwoCL$ and $\ltwostar$}

We will regard the pairing on $\ltwoCL$ as giving an isomorphism:
\[
\ltwoCL \xrightarrow{\sim} \ltwostar: \, f \mapsto \Big< f , \, \emptydot \Big>
\]
The operator norm of 
\[
\Big< f , \, \emptydot \Big> \in \ltwostar
\]
is 
\[
\sup_{||h||=1} \left|\Big< f , \, h \Big>\right| = \sqrt{\Big< f , \, f \Big>}
\]
Hence, the map
\[
\ltwoCL \rightarrow \ltwostar
\]
is norm preserving. Moreover, the inverse of this map sends $\ltwostar$ to 
\[
\ltwoCL \cong \ltwoCL^{**}
\]
and hence defines a symmetric pairing on $\ltwostar$. The norm induced by this pairing agrees with the \textit{operator norm} on $\ltwostar$. Thus, $\ltwostar$ also has the structure of a Hilbert space.

\subsubsection{Motivation for using $\ltwostar$}

The main reason that we would like to work with $\ltwostar$ is that the elements of $\ltwostar$ are measures, not functions, and that our operator $\Delta_0$ is phrased most naturally as \textit{an operator on measures}.

\begin{itemize}
\item[]\begin{remark}
$\Delta_0$ was originally defined on measures. There are two natural ways to extend it to $\ltwoCL$. Firstly, we can extend it to $\ltwoCL$ via:
\[
\ltwostar \cong \ltwoCL: \, f \mapsto \Big< f, \et \Big>
\]
Secondly, we can take the dual of $\Delta_0$. The equivalence of both definitions is implied by, and in fact equivalent to, the fact that $\Delta_0$ is self-adjoint. We will denote the resulting operator on $\ltwoCL$ as $\Delta_0^{*}$.
\end{remark}
\end{itemize}
With this in mind, it might have perhaps been wise to state the results in the introduction in terms of $\ltwostar$ immediately. However, this would have made notation heavier than necessary.
\textit{Results are easier to state in $\ltwoCL$, but easier to prove in $\ltwostar$.}

It is important to keep in mind that the two spaces are canonically isomorphic; hence that choice of which space to work with is merely a matter of convenience.

\subsubsection{Things to keep in mind when passing between $\ltwoCL$ and $\ltwostar$}

Given $G \in X_0$ we write $\nu(G)$ to denote the measure $\nu$ evaluated against the indicator function of $G$. In other words $\nu(G)$ is the weight (or probability, if $\nu$ is a probability measure) of $G$. Thus if $\nu = \Big<f,\, \emptydot \Big>$, then
\[
\nu(G) = \mu_0(G)f(G)
\]

Thus, the difference between $\ltwoCL$ and $\ltwostar$ amounts to a different choice of basis.

\subsubsection{Examples}

We collect here a dictionary between the two spaces:
\begin{center}
   \begin{tabular}{|c|c|}
   \hline
 $\ltwoCL$  & $\ltwostar$  \\ \hline \hline
 Typical element: $f$  & Typical element:  $\nu$  \\ \hline
 $1$  & $\mu_0$  \\ \hline
 $\#Sur(\emptydot,F)$  & $\mu_0(\emptydot)\#Sur(\emptydot,F)$  \\ \hline
\hspace{0.2in}$\Big<f_1,f_2\Big>\defeq \sum_G \mu_0(G)f_{1}(G)f_{2}(G)$\hspace{0.2in}  & \hspace{0.2in}$\Big<\nu_1,\nu_2\Big>\defeq\sum_G \frac{\nu_{1}(G)\nu_{2}(G)}{\mu_0(G)}$ \hspace{0.2in} \\ \hline
$\Big<1,1\Big>=1$ & $\Big< \mu_0 , \mu_0 \Big>=1$  \\ \hline
$\Delta_0^{*}$ & $\Delta_0$  \\ \hline
   \end{tabular}
\end{center}

\begin{mysection}
\[
   \begin{mat}{|c|c|}
   \hline
 \ltwoCL  & \ltwostar  \\ \hline
 f  & \nu \\ \hline
 1  & \mu_0  \\ \hline
 \#Sur(\emptydot,F)  & \mu_0(\emptydot)\#Sur(\emptydot,F)  \\ \hline
 \Big<f_1,f_2\Big>\defeq \sum_G \mu_0(G)f_{1}(G)f_{2}(G)  & \Big<\nu_1,\nu_2\Big>\defeq  \\ 
 \sum_G \mu_0(G)f_{1}(G)f_{2}(G)  &  \sum_G \frac{f_{1}(G)f_{2}(G)}{\mu_0(G)} \\ \hline
   \end{mat}
\]
\end{mysection}

\begin{mysection}
The main reason that we would like to work with $\ltwostar$ is that the elements of $\ltwostar$ are measures, not functions, and that our operator $\Delta_0$ is phrased most naturally as \textit{an operator on measures}. The same is true for other operators we use in this chapter.

\begin{remark}
$\Delta_0$ is a priori defined on $\ltwostar$. However, as it is a bounded operator, we can extend it to $\ltwoCL$ via the isometry between $\ltwostar$ and $\ltwoCL$.
\end{remark}

With this in mind, it might have perhaps been wise to state the results in the introduction in terms of $\ltwostar$ immediately. However, this would have made notation heavier than necessary.

\textit{
Results are easier to state in $\ltwoCL$, but easier to prove in $\ltwostar$.}

It is important to keep in mind that the two spaces are canonically isomorphic; hence that choice of which space to work with is merely a matter of convenience.

The pairing on $\ltwoCL$ is a map:
\[
\ltwoCL \xrightarrow{\sim} \ltwostar: f \mapsto \nu
\]
defined by
\[
f \mapsto f(\emptydot)\mu_0(\emptydot)=c_0\frac{f(\emptydot)}{\#Aut(\emptydot)} 
\]
The inverse map 
\[
\ltwostar \xrightarrow{\sim} \ltwoCL \cong \ltwostar{}^{*} : \nu \mapsto f
\]
 is hence given by
 \[
 \nu \mapsto \frac{\nu(\emptydot)}{\mu_0(\emptydot)}=
 c_0^{-1} \nu(\emptydot)\#Aut(\emptydot)
 \]
 This gives the following pairing on $\ltwostar$:
 \[
 \Big<\nu_1, \, \nu_2 \Big> = c_0^{-1} \sum_G \#Aut(G)|\nu_1(G)||\nu_2(G)| 
 \]
 .....The isomorphism between the two spaces is an isometry....
.....We record a dictionary between the two Hilbert spaces....
\end{mysection}

\end{mymysection}

\bibliographystyle{alpha}
\bibliography{ThesisBibliography}

\end{document}